\documentclass[12pt]{amsart}
\usepackage{amssymb}
\usepackage{enumerate}
\usepackage{hyperref,hyphenat}
\usepackage[all]{xy}
\usepackage[toc,page,title,titletoc,header]{appendix}
\usepackage{xcolor}

\numberwithin{equation}{section}

\theoremstyle{plain}
\newtheorem{Thm}{Theorem}[section]
\newtheorem{Prop}[Thm]{Proposition}
\newtheorem{Lem}[Thm]{Lemma}
\newtheorem{Que}[Thm]{Question}
\newtheorem{Cor}[Thm]{Corollary}
\newtheorem{Conj}[Thm]{Conjecture}

\theoremstyle{remark}
\newtheorem{Rem}[Thm]{Remark}
\theoremstyle{definition}
\newtheorem{Def}[Thm]{Definition}
\newtheorem{Eg}[Thm]{Example}

\newcommand{\Mod}[1]{\ (\mathrm{mod}\ #1)}

\def\A{\mathbb{A}}
\def\Z{\mathbb{Z}}
\def\Q{\mathbb{Q}}
\def\R{\mathbb{R}}
\def\C{\mathbb{C}}
\def\sM{\mathcal{M}}
\def\sO{\mathcal{O}}
\def\P{\mathbb{P}}
\def\PGL{\mathrm{PGL}}
\def\Aff{\mathrm{Aff}}
\def\Spec{\mathrm{Spec}}
\def\Per{\mathrm{Per}}
\def\PrePer{\mathrm{PrePer}}
\def\Aut{\mathrm{Aut}}
\def\Rat{\mathrm{Rat}}
\def\FL{\mathrm{FL}}
\def\Id{\mathrm{Id}}
\def\Poly{\mathrm{Poly}}
\def\MPoly{\mathrm{MPoly}}
\def\PNI{\mathrm{PNI}}
\def\Pcrit{\mathrm{Pcrit}}
\def\MPcrit{\mathrm{MPcrit}}
\def\Mon{\mathrm{Mon}}
\def\Fix{\mathrm{Fix}}
\def\ord{\mathrm{ord}}
\def\sac{\mathrm{sac}}
\def\multi{\mathrm{multi}}
\def\lcm{\mathrm{lcm}}
\def\critV{\mathrm{CritV}}
\def\Res{\mathrm{Res}}
\def\min{\mathrm{min}}
\def\la{\lambda}

\begin{document}

\title[]{On the multiplier spectrum of polynomials}

\author{Geng-Rui Zhang}

\address{School of Mathematical Sciences, Peking University, Beijing 100871, China}

\email{grzhang@stu.pku.edu.cn, chibasei@163.com}
\date{February 4, 2026}



\maketitle

\begin{abstract}
	We prove several results on the multiplier spectrum of polynomials. We provide a detailed proof of the theorem stating that the multiplier spectrum morphism is generically injective on the moduli space of polynomials. We obtain a description of the non-injective locus of the multiplier spectrum morphism for polynomials of degree $d\geq2$. Roughly speaking, we prove that, apart from isolated exceptions, polynomials with the same multiplier spectrum are intertwined. More precisely, we show that, up to iteration and isolated exceptions, the polynomials are either equivalent or related by Ritt moves. We also investigate the relationship between Ritt moves and multiplier spectra over arithmetic progressions.
\end{abstract}
	
\tableofcontents
	
\section{Introduction}
This paper aims to study the multipliers at periodic points of a rational map $f:\P^1(\C)\to\P^1(\C)$ of degree at least $2$, with a specific focus on the case where $f$ is a polynomial.

\subsection{The multiplier spectrum morphism on rational maps}\label{S1.1}
We recall some general definitions, constructions, and results for rational maps (not just for polynomials).

Let $f:\P^1\to\P^1$ be an endomorphism of degree $d\geq2$ over $\C$, i.e., $f\in\C(z)$ is of degree $d\geq2$. For an $f$-periodic point $z_0\in\P^1(\C)$ with exact period $n$, set $n_f(z_0):=n$. The \emph{multiplier} of $f$ at $z_0$ is defined as the differential
$$\rho_f(z_0):=df^{\circ n}(z_0)\in\C.$$
(Throughout this paper, the symbol $\circ$ is always used to denote composition of rational maps, while its absence indicates multiplication in the field $\C(z)$ of rational maps. ) Write $n(z_0)=n_f(z_0)$ and $\rho(z_0)=\rho_f(z_0)$ for simplicity, when the map $f$ is clear from context. The \emph{length} of $f$ at $z_0$ is the modulus $\left|\rho_f(z_0)\right|$. The \emph{characteristic exponent} of $f$ at $z_0$ is defined to be $\chi_f(z_0):=n^{-1}\log\left|\rho_f(z_0)\right|$, when $\rho_f(z_0)\neq0$. Note that multipliers, lengths, and characteristic exponents are invariant under conjugacy by M\"obius transformations.

Let $\Per(f)(\C)$ denote the set of $f$-periodic points in $\P^1(\C)$,
and set
$$\Per^*(f)(\C):=\{z_0\in\Per(f)(\C)\colon\rho_f(z_0)\neq0\}.$$
Write $\Per(f)$ and $\Per^*(f)$ for simplicity when the base field $\C$ is clear. Let $\PrePer(f)(\C)$ (or $\PrePer(f)$) denote the set of $f$-preperiodic points in $\P^1(\C)$.

We identify $\PGL_2(\C)$ with the group of rational maps $g\in\C(z)$ of degree $1$ (i.e., M\"obius transformations over $\C$), which forms the automorphism group of the projective line $\P^1$ over $\C$.

For endomorphisms $g:X\to X$ and $h:Y\to Y$ on algebraic varieties, we say that $h$ is \emph{semi-conjugate} to $g$ if there is a dominant morphism $\pi:X\to Y$ such that $\pi\circ g=h\circ\pi$, written $g\geq h$ (or $g\geq_\pi h$ when $\pi$ is specified).

In complex dynamics, the exceptional maps defined below are often regarded as special examples within the class of rational maps of degree $\geq2$. They are the rational maps on $\P^1(\C)$ that are related to algebraic groups and exhibit special dynamical properties.

\begin{Def}\label{excep}
	Let $f:\P^1\to\P^1$ be an endomorphism over $\C$ of degree $d\geq2$. 
\begin{itemize}
	\item The map $f$ is called a \emph{Latt\`es map} if there exists an endomorphism $\phi$ on an elliptic curve $E$ such that $\phi\geq f$. The map $f$ is called \emph{flexible Latt\`es} if there exist an elliptic curve $E$ and $n\in\Z\setminus\{0,\pm1\}$ such that $[n]\geq_\pi f$, where $[n]$ denotes the multiplication-by-$n$ map on $E$ and $\pi:E\to\P^1$ is the quotient map modulo $\{\pm1\}$. (Note that in this case we must have $d=n^2$, a perfect square in $\Z$.) A non-flexible Latt\`es map is called \emph{rigid Latt\`es}.
	\item We say that $f$ is of \emph{monomial type} if it is semi-conjugate to the power map $z\mapsto z^n$ on $\P^1$ for some $n\in\Z\setminus\{0,\pm1\}$. It is well-known that $f$ is of monomial type if and only if it is conjugate to $z^{\pm d}$ or $\pm T_d(z)$, where $T_d(z)$ is the (normalized) \emph{Chebyshev polynomial of degree $d$} (or the Dickson polynomial of degree $d$ with parameter $1$), i.e., the unique monic polynomial $T_d\in\C[z]$ of degree $d$ such that $T_d(z+z^{-1})=z^d+z^{-d}$. See \cite[Chapter~6]{Silverman2007} and \cite{milnor2006lattes}. 
	\item $f$ is called \emph{exceptional} if it is Latt\`es or of monomial type.
\end{itemize}
These definitions depend only on the conjugacy class of $f$. Moreover, it is well-known that $f$ is exceptional if and only if some iterate (or every iterate) $f^{\circ k}$ is exceptional ($k\in\Z_{>0}$).
\end{Def}

Fix an integer $d\geq2$. Let $\Rat_d$ be the space of degree $d$ endomorphisms on $\P^1$, which is a smooth affine variety of dimension $2d+1$ \cite{Silverman2012}. 
Let $\FL_d\subseteq\Rat_d$ be the locus of flexible Latt\`es maps, which is Zariski closed in $\Rat_d$. 
The group $\PGL_2=\Aut(\P^1)$ acts on $\Rat_d$ by conjugacy. The geometric quotient
$$\sM_d:=\Rat_d/\PGL_2=\Spec (\sO(\Rat_d)^{\PGL_2})$$
is the (coarse) \emph{moduli space of endomorphisms of degree $d$} \cite{Silverman2012}, which is an affine variety of dimension $2d-2$ \cite[Theorem~4.36(c)]{Silverman2007}. Let $\Psi:\Rat_d\to\sM_d$ be the quotient morphism. For $f\in\Rat_d(\C)$, we also denote its conjugacy class by $\Psi(f)=[f]\in\sM_d(\C)$. Set $[\FL_d]:=\Psi(\FL_d)$. Note that $\FL_d$ is non-empty if and only if $d$ is a perfect square in $\Z$, and in this case $[\FL_d]$ is equidimensional of dimension $1$. The above construction works over any algebraically closed field of characteristic zero and commutes with base changes.

Fix $n\in\Z_{>0}$. Let $\Per_n(f)=\Per_n(f)(\C)$ be the multiset
\footnote{A \emph{multiset} is a collection that allows multiple instances for each of its elements. The number of instances of an element is called its \emph{multiplicity}. For example, $\{a,a,b,c,c,c\}$ is a multiset of cardinality $6$, with multiplicities $2,1,3$ for $a,b,c$, respectively.}
of fixed points of $f^{\circ n}$. This multiset has cardinality $d^n+1=:N_{d,n}$. Using elementary symmetric polynomials, the multipliers at the $N_{d,n}$ fixed points of $f^{\circ n}$ define a point
$$S_n(f)=(\sigma_{1,n}(f),\dots,\sigma_{N_{d,n},n}(f))\in\C^{N_{d,n}},$$
where $\sigma_{j,n}(f)$ is the $j$-th elementary symmetric polynomial in the $N_{d,n}$-tuple $\{\rho_{f^{\circ n}}(z)\}_{z\in\Per_n(f)}$ for $1\leq j\leq N_{d,n}$. We define the \emph{multiplier spectrum} of $f$ to be the sequence $S(f)=(S_n(f))_{n=1}^\infty$. It is clear that $S_n$ takes the same value within a conjugacy class of rational maps and thus defines a morphism $S_n:\sM_d\to\A^{N_{d,n}}$ over $\C$, called the \emph{multiplier spectrum morphism of level $n$}. (In fact, the morphism $S_n$ is defined over $\Q$.) The morphism
$$\tau_{d,n}:\sM_d\to\A^{N_{d,1}}\times\cdots\times\A^{N_{d,n}}=\A^{N_{d,1}+\cdots+N_{d,n}}$$
given by $\tau_{d,n}([f])=\left(S_1([f]),\dots,S_n([f])\right)$ is called the \emph{multiplier spectrum morphism up to level $n$}. For $f\in\Rat_d(\C)$, we also write $\tau_{d,n}(f):=\tau_{d,n}([f])$.

For $n\in\Z_{>0}$, set $R_{d,n}:=\{([f],[g])\in\sM_d(\C)^2\mid\tau_{d,n}(f)=\tau_{d,n}(g)\}$. Then $(R_{d,n})_{n=1}^\infty$ forms a decreasing sequence of Zariski closed subsets of $\sM_d(\C)^2$, which stabilizes due to the Noetherian property. Let $m_d\geq1$ be minimal such that $\cap_{n=1}^\infty R_{d,n}=R_{d,m_d}$. Consequently, for all $f,g\in\Rat_d(\C)$, $f$ and $g$ have the same multiplier spectrum $S(f)=S(g)$ if and only if $\tau_{d,m_d}(f)=\tau_{d,m_d}(g)$, i.e., $S_n(f)=S_n(g)$ for $1\leq n\leq m_d$. Define $\tau_d:=\tau_{d,m_d}:\sM_d\to\A^{N_{d,1}+\cdots+N_{d,m_d}}$, called the \emph{multiplier spectrum morphism} on $\sM_d$. Abusing notation, we also use $\tau_d$ to denote $\tau_d\circ\Psi$, called the \emph{multiplier spectrum morphism} on $\Rat_d$. For two rational maps $f,g\in\Rat_d(\C)$, $f$ and $g$ are called \emph{isospectral} if they have the same multiplier spectrum $S(f)=S(g)$ (or equivalently, if $\tau_d(f)=\tau_d(g)$). It is well-known that rational maps within an irreducible component of $\FL_d(\C)$ are isospectral.

The following remarkable theorem of McMullen \cite{McMullen1987} establishes that, with the exception of flexible Latt\`es maps, the multiplier spectrum determines the conjugacy class of rational maps up to only finitely many possibilities.

\begin{Thm}[McMullen]
	For every integer $d\geq2$, the morphism
	$$\tau_d:\sM_d(\C)\setminus\Psi(\FL_d(\C))\to\A^{N_{d,1}}(\C)\times\cdots\times\A^{N_{d,m_d}}(\C)$$
	is quasi-finite.
\end{Thm}

Recently, Ji and Xie proved a significant generalization \cite[Theorem~1.3]{JX23} of McMullen's theorem:
\begin{Thm}[Ji and Xie]\label{ratgeninj}
	For every integer $d\geq2$, the morphism
	$$\tau_d:\sM_d(\C)\to\A^{N_{d,1}}(\C)\times\cdots\times\A^{N_{d,m_d}}(\C)$$
	is generically injective; that is, there exist a non-empty Zariski open subset $U$ of $\sM_d(\C)$ and a Zariski open subset $W$ of the Zariski closure of $\tau_d(U)$ with $\tau_d^{-1}(W)=U$ such that $\tau_d:U\to W$ is a finite morphism of degree $1$.
\end{Thm}

Similarly, by replacing multipliers with lengths and $\C$ with $\R$ in the above definitions, we obtain for each $n\in\Z_{>0}$ the \emph{length spectrum map} $L_n:\sM_d(\C)\to\R^{N_{d,n}}$ of level $n$. The \emph{length spectrum} of $f$ is defined as the sequence $L(f)=(L_n(f))_{n=1}^\infty$. We define the \emph{length spectrum map up to level $n$} as
$$\eta_{d,n}:\sM_d(\C)\to\R^{N_{d,1}}\times\cdots\times\R^{N_{d,n}},[f]\mapsto(L_1(f),\dots,L_n(f)).$$
(For the time being, we treat these maps merely as set-theoretic maps.)

\subsection{Statement of the main results}
Fix an integer $d\geq2$. By restriction, we obtain the multiplier spectrum morphism for polynomials. Let $\Poly^d\subset\Rat_d$ be the closed subvariety of polynomials of degree $d$, and let $\MPoly^d:=\Poly^d/\Aff$ be the \emph{moduli space of polynomials of degree $d$}. See \cite{FG22}. Here $\Aff$ is the group of automorphisms of $\A^1$ consisting of all degree one polynomials, acting on $\Poly^d$ by conjugacy. So $\Aff(\C)=\{az+b\in\C[z]\colon a\in\C^\times,b\in\C\}$. The following conditions are equivalent for all polynomials $f,g\in\Poly^d(\C)$:
\begin{itemize}
	\item there exists $\tau\in\PGL_2(\C)$ such that $f=\tau\circ g\circ\tau^{-1}$;
	\item there exists $\sigma\in\Aff(\C)$ such that $f=\sigma\circ g\circ\sigma^{-1}$;
\end{itemize}
This allows us to view $\MPoly^d\subset\sM_d$ as a subset of $\sM_d$. We abuse notation and let $\Psi:\Poly^d\to\MPoly^d$ still denote the quotient morphism. For every $f\in\Poly^d(\C)$, we also use $\Psi(f)=[f]\in\MPoly^d(\C)$ to denote the (affine) conjugacy class of $f$ in $\Poly^d(\C)$. The inclusions $\MPoly^d\stackrel{\iota_0}{\hookrightarrow}\sM_d$ and $\Poly^d\stackrel{\iota}{\hookrightarrow}\Rat_d$ are compatible with $\Psi$, i.e., $\Psi\circ\iota=\iota_0\circ\Psi$. For every integer $n\geq1$, denote the restriction of $\tau_{d,n}$ (resp. $\tau_d$) to the moduli space $\MPoly^d$ of polynomials by $\tilde{\tau}_{d,n}$ (resp. $\tilde{\tau}_d$). Similarly, we obtain $\tilde{\eta}_{d,n}$ for the length spectrum.

\subsubsection*{Generic injectivity}
Ji and Xie also proved the following polynomial version of Theorem~\ref{ratgeninj}, which was also proved independently by Huguin \cite{VH2412} using completely different methods:
\begin{Thm}[{\cite[Theorem~1.4]{JX23} and \cite[Theorem~C]{VH2412}}]\label{polygeninj}
	For every integer $d\geq2$, the morphism
	$$\tilde{\tau}_d:\MPoly^d(\C)\to\A^{N_{d,1}}(\C)\times\cdots\times\A^{N_{d,m_d}}(\C)$$
	is generically injective.
\end{Thm}
The proof of Theorem~\ref{polygeninj} in \cite{JX23} parallels that of Theorem~\ref{ratgeninj} in the same paper. Specifically, the proof of Theorem~\ref{ratgeninj} relies on two key ingredients:
\begin{itemize}
	\item a DAO-type theorem \cite[Theorem~3.4]{JX23} for curves (here ``DAO'' means ``Dynamical Andr\'e-Oort'');
	\item a result \cite[Theorem~3.3]{JX23} about ``general'' rational maps, which is based on Pakovich's work \cite{Paksimple}.
\end{itemize}
In contrast, the proof of Theorem~\ref{polygeninj} in \cite{JX23} just replaces the second ingredient \cite[Theorem~3.3]{JX23} with \cite[Theorems~3.51~and~3.52]{FG22}. Huguin's proof \cite{VH2412} of Theorem~\ref{polygeninj} actually shows that $\tilde{\tau}_{d,2}$ is generically injective for every $d\geq2$, and it is based on intricate local computations.

The class of simple rational maps is involved in the second ingredient \cite[Theorem~3.3]{JX23}. A rational map $G\in\C(z)$ of degree $d\geq2$ is called \emph{simple} if $G$ has exactly $2d-2$ distinct critical values in $\P^1(\C)$. However, a polynomial of degree $d\geq3$ is never simple, due to the totally invariant fixed point $\infty$. To overcome this problem, for a polynomial $F\in\C[z]$ of degree $d\geq2$, we call $F$ \emph{pre-simple} if it has exactly $d-1$ distinct critical values in $\C$. This class of pre-simple polynomials serves as an analogue of simple rational maps for the polynomial case. We observe that Pakovich's methods \cite{Paksimple} are applicable to pre-simple polynomials. In \S~\ref{secpresimple}, we present the following polynomial version of \cite[Theorem~3.3]{JX23} (for the case $d\geq4$), which is essentially due to Pakovich \cite{Paksimple}:
\begin{Thm}\label{poly33S1}
	For every integer $d\geq4$, there exists a non-empty Zariski open subset $U$ of $\Poly^d(\C)$ such that, for every $f\in U$, the following hold:
	\begin{enumerate}[(1)]
		\item $f$ is pre-simple;
		\item for every pre-simple $g\in\Poly^d(\C)$, $f$ and $g$ are intertwined if and only if $[f]=[g]$ in $\MPoly^d(\C)$.
	\end{enumerate}
	Here, two polynomials $F,G\in\C[z]$ of degree $\geq2$ are called \emph{intertwined} if there exists a (possibly reducible) algebraic curve $Z\subset\A^2$ whose projections to both axes are onto such that $Z$ is invariant under the endomorphism $$(F,G):\A^2\to\A^2,(x,y)\mapsto(F(x),G(y)).$$
\end{Thm}
\begin{Rem}\label{interwdef}
	Let $F,G\in\C[z]$ be polynomials of degree $\geq2$. It is known that $F$ and $G$ are intertwined if and only if there exist non-constant polynomials $R,h_1,h_2\in\C[z]\setminus\C$ and $n\in\Z_{>0}$ such that
	$$F^{\circ n}\circ h_1=h_1\circ R\quad\text{and}\quad G^{\circ n}\circ h_2=h_2\circ R.$$
	In particular, $\deg(F)=\deg(G)$. See \cite[Theorem~3.39]{FG22}.
\end{Rem}
Recently, Pakovich \cite{Pakdeg23} studied general rational maps of degree $2\leq d\leq3$ in detail. Consequently, one could get a proof of Theorem~\ref{ratgeninj} for this case, using the results in \cite{Pakdeg23} as part of the argument. The philosophy is that a general rational map of degree $d$ has ``good'' properties regarding the decompositions of its iterates. However, the simplicity condition (used in \cite{Paksimple} where $d\geq4$) is not sufficient for the cases $2\leq d\leq3$.

In \S~\ref{secdeg23}, following Pakovich \cite{Pakdeg23}, we study general polynomial maps of degree $2\leq d\leq3$ and obtain the following polynomial version of \cite[Theorem~3.3]{JX23} for these lower degree cases:
\begin{Prop}\label{poly33deg23intro}
	Let $d\in\{2,3\}$. There exists a non-empty Zariski open subset $U$ of $\Poly^d(\C)$ such that for all $(f,g)\in U\times U$, $f$ and $g$ are intertwined if and only if $[f]=[g]$ in $\MPoly^d(\C)$.
\end{Prop}
Proposition~\ref{poly33deg23intro} is also a direct corollary of \cite[Theorems~3.51~and~3.52]{FG22} (see also \cite[Theorem~1.4]{GNY19}). We include this proposition to present different approaches.

In \S~\ref{S2.3}, we provide a detailed proof of Theorem~\ref{polygeninj} using Theorem~\ref{poly33S1} and Proposition~\ref{poly33deg23intro}, which is essentially due to Pakovich \cite{Paksimple,Pakdeg23} and Ji-Xie \cite{JX23}.

The intermediate results presented in \S~\ref{S2} may be useful for studying the dynamics of one-variable polynomials and their products. For instance, we present a proof of the Zariski-dense orbit conjecture for a general split polynomial endomorphism on $\A^2$ with all factors of degree $d\geq2$ (see Theorem~\ref{presimapp}).

\subsubsection*{The non-injective locus}
After establishing the generic injectivity of the multiplier spectrum morphism, a natural further direction is to study the non-injective locus of $\tau_d$ (or $\tilde{\tau}_d$), i.e., the locus
$$\{([f],[g])\in\sM_d(\C)^2\mid\tau_{d,m_d}(f)=\tau_{d,m_d}(g)\}\setminus\Delta,$$
or its intersection with $\MPoly^d(\C)^2$, where $\Delta$ is the diagonal in $\sM_d(\C)^2$.

Outside the flexible Latt\`es locus, there are two known cases of non-conjugate isospectral rational maps. For a non-constant rational map $f\in\C(z)\setminus\C$, if $f=h_1\circ h_2$ for some $h_1,h_2\in\C(z)$, then $h:=h_2\circ h_1$ is called an \emph{elementary transformation} of $f$. Two non-constant rational maps $f$ and $g$ are called \emph{equivalent}, written $f\sim g$, if there exists a finite chain of elementary transformations between $f$ and $g$. Clearly, equivalence defines an equivalence relation on $\C(z)\setminus\C$, which is coarser than conjugacy (i.e., conjugacy implies equivalence). So we can talk about the equivalence of two conjugacy classes of rational maps. An easy lemma of Pakovich \cite[Lemma~2.1]{Pak19b} shows that if two rational maps $f$ and $g$ are equivalent (of degree $\geq2$), then they are isospectral. The other known case of non-injectivity comes from exceptional maps. Silverman \cite[Theorem~6.62]{Silverman2007} showed that $\tau_d$ is not injective at $[f]\in\sM_d(\C)$ for some rigid Latt\`es $f$. Pakovich asked \cite[Problem~3.1]{Pak19a} whether these are the only obstructions to the injectivity of $\tau_d$. Ji and Xie \cite[Conjecture~1.5]{JX23} conjectured that Pakovich's question has a positive answer:
\begin{Conj}\label{ratnoninj}
Let $f$ and $g$ be non-conjugate isospectral rational maps of degree $d\geq2$. Then one of the following holds:
\begin{enumerate}[(i)]
	\item \label{ratnoninj1} both $f$ and $g$ are Latt\`es maps;
	\item \label{ratnoninj2} $f\sim g$.
\end{enumerate}
\end{Conj}
Since polynomials are not Latt\`es, we only need to consider the case \eqref{ratnoninj2} in Conjecture~\ref{ratnoninj}. Note that if the degree $d$ is a prime number, then two rational maps $f$ and $g$ of degree $d$ are equivalent if and only if they are conjugate. Consequently, Conjecture~\ref{ratnoninj} predicts that for every prime $p$, the multiplier spectrum morphism $\tilde{\tau}_p$ is injective on $\MPoly^p(\C)$. On the other hand, an observation of Huguin \cite[Corollary~85]{VH2412} shows that $\tilde{\tau}_d$ is not injective on $\MPoly^d(\C)$ for every composite $d\geq2$.

If the degree $d$ is small, then it is possible to analyze the non-injective locus of $\tilde{\tau}_d$ (indeed, $\tilde{\tau}_{d,m}$ with small $m$) using explicit computations. For prime degree $2\leq d\leq3$, Huguin \cite[Theorem~C]{VH2412} showed that $\tilde{\tau}_{d,1}$ is injective on the moduli space $\MPoly^d(\C)$. Recall that Milnor \cite{Milnor1993} showed that $\tau_{2,1}$ is injective on $\Rat_2(\C)$. For the composite degree $d=4$, a result \cite[Proposition~86]{VH2412} of Huguin (whose proof relies on computer computations) shows that the following are equivalent for all $(f,g)\in\Poly^4(\C)^2$:
\begin{itemize}
	\item $\tilde{\tau}_4(f)=\tilde{\tau}_4(g)$;
	\item $\tilde{\tau}_{4,2}(f)=\tilde{\tau}_{4,2}(g)$;
	\item $[f]=[g]$ in $\MPoly^4(\C)$ or $(f,g)=(h_1\circ h_2,h_2\circ h_1)$ for some $(h_1,h_2)\in\Poly^2(\C)^2$ (in particular, $f\sim g$).
\end{itemize}

We investigate the non-injective locus of the multiplier spectrum morphism $\tilde{\tau}_d$ for polynomials. In \S~\ref{S3}, for all $d\geq2$, we provide a first description of the non-injective locus of $\tilde{\tau}_d$, i.e., the locus
$$\PNI_d:=\{([f],[g])\in\MPoly^d(\C)^2\mid\tilde{\tau}_d([f])=\tilde{\tau}_d([g]),[f]\neq[g]\},$$
which serves as a starting point for Conjecture~\ref{ratnoninj} for the polynomial case:
\begin{Thm}\label{intertwined}
	Let $d\geq2$ be an integer, and let $C\subset\MPoly^d\times\MPoly^d$ be an irreducible curve defined over $\overline{\Q}$ such that the projection of $C$ to each factor $\MPoly^d$ is non-constant. Assume that $C(\C)\setminus\PNI_d$ is finite. Then for every $t=([f_t],[g_t])\in C(\C)$, $f_t$ and $g_t$ are intertwined.
\end{Thm}
In fact, we prove a more precise version:
\begin{Thm}\label{polynoninj}
	Let $\phi_1,\phi_2:C\times\P^1\to C\times\P^1$ be two non-isotrivial (algebraic) families over $\overline{\Q}$ of degree $d$ ($d\geq2$) polynomials, parametrized by an affine irreducible curve $C$ over $\overline{\Q}$. Denote the induced morphisms also by $\phi_1,\phi_2:C\to\Poly^d$. Assume that $([\phi_1(t)],[\phi_2(t)])\in\PNI_d$ for all but finitely many $t\in C(\C)$. (Here, non-isotriviality means that $[\phi_j(C)]$ is not a single point in $\MPoly^d$ for $1\leq j\leq2$.)
	
	Then there exists a finite subset $S$ of $C(\C)$ such that, for every $t\in C(\C)\setminus S$, one of the following conditions holds:
	\begin{enumerate}[(1)]
		\item \label{polynoninj1} there exists $N\in\Z_{>0}$ such that $\phi_1(t)^{\circ N}\sim\phi_2(t)^{\circ N}$;
		\item \label{polynoninj2} there exist $N,k_1,k_2,l\in\Z_{>0}$ and $V(z)\in\C[z]\setminus\C$, with
		$$k_1\neq k_2,\gcd(d,k)=1,l\equiv1\Mod{k},V(0)\neq0$$
		such that
		$$\phi_1(t)^{\circ N}\sim z^l\cdot V(z^{k_1})^{k_1^\prime}\quad\text{and}\quad\phi_2(t)^{\circ N}\sim z^l\cdot V(z^{k_2})^{k_2^\prime},$$
		where $k:=\lcm(k_1,k_2)$ denotes the least common multiple of $k_1$ and $k_2$, and $k_j^\prime=k/k_j$ for $1\leq j\leq2$.
	\end{enumerate}
\end{Thm}
\begin{Rem}
	In Theorem~\ref{polynoninj} \eqref{polynoninj2}, we have
	$$(\phi_1(t)^{\circ N})_\min\sim (\phi_2(t)^{\circ N})_\min;$$
	see Remark~\ref{Remminequiv} for the precise meaning.
	
	In \eqref{polynoninj2}, if $k_j^\prime>1$ ($1\leq j\leq2$), then $z^l\cdot V(z^{k_j})^{k_j^\prime}$ is a generalized Latt\`es map in the sense of Pakovich \cite{Pakgenla} (see Definition~\ref{genla} and Proposition~\ref{polygenla}). Note that at least one $k_j^\prime>1$ because $k_1\neq k_2$.
\end{Rem}
\begin{Rem}\label{polynoninjRem}
Let $K$ be a number field such that $\phi_1$, $\phi_2$, and $C$ are defined over $K$. Then in Theorem~\ref{polynoninj} \eqref{polynoninj1} or \eqref{polynoninj2}, we can choose $N$ to be bounded by a constant depending only on $[K:\Q]$ and $d$; see Remark~\ref{Remmainpf}.
\end{Rem}
The proof of Theorem~\ref{polynoninj} is divided into two steps. The first step is as follows. After a base change, we may assume that all the critical points of $\phi_1$ and $\phi_2$ are marked. It is well-known that postcritically finite (PCF) parameters do not form families in $\MPoly^d$. We can argue that the multiplier spectrum determines the periodicity of critical points; see Proposition~\ref{multidet} for a precise statement. Then from $\phi_1$ and $\phi_2$ we can construct two entangled active dynamical pairs in the sense of Favre and Gauthier \cite{FG22}, still denoted by $\phi_1$ and $\phi_2$ parametrized by $C$. Applying a theorem of Favre and Gauthier (Theorem~\ref{FG}), there exists a family $R$ of polynomials parametrized by a non-empty Zariski open subset $U\subseteq C$ that is semi-conjugate to $\phi_1^{\circ N}$ and $\phi_2^{\circ N}$, for some integer $N\geq1$. In the second step, we need some results of Pakovich on generalized Latt\`es maps \cite{Pakgenla}. Set $S=(C\setminus U)(\C)$, which is a finite set. Fix $t\in U(\C)$. Assume that \eqref{polynoninj1} does not hold for $t$. Then a result of Pakovich \cite{Pakgenla} (see Theorem~\ref{genla32}) implies that $R(t)$ is a generalized Latt\`es map. By the classification of generalized Latt\`es maps that are polynomials (Proposition~\ref{polygenla}) and some elementary arguments, we conclude that \eqref{polynoninj2} holds for $t$, after possibly further iteration.

\subsection{Organization of the article}
In \S~\ref{SPak}, we provide a brief introduction to orbifolds and generalized Latt\`es maps based on \cite{Pakgenla}, as needed for \S~\ref{S2}--\ref{S3}.

In \S~\ref{S2}, we explore the generic injectivity of the multiplier spectrum morphism for polynomials. In \S~\ref{secpresimple}, we introduce the notion of pre-simple polynomials and study their properties, which are essentially done by Pakovich \cite{Paksimple}. Then we deduce Theorem~\ref{poly33S1}. In \S~\ref{secdeg23}, we study general polynomials of degree $2\leq d\leq3$ and obtain Proposition~\ref{poly33deg23intro}, following the ideas in \cite{Pakdeg23}. In \S~\ref{S2.3}, we deduce Theorem~\ref{polygeninj} from Theorem~\ref{poly33S1} and Proposition~\ref{poly33deg23intro}, following Ji and Xie \cite{JX23}. Finally, in \S~\ref{secZDO}, we present a proof of the Zariski-dense orbit conjecture for a general split polynomial endomorphism on $\A^2$ with all factors of degree $d\geq2$ (Theorem~\ref{presimapp}).

In \S~\ref{S3.1}, we recall the notion of critically marked polynomials and state a theorem of Favre and Gauthier (Theorem~\ref{FG}), which is essential for the proof of Theorem~\ref{polynoninj}. Then we proceed to complete the proofs of Theorems~\ref{polynoninj}~and~\ref{intertwined} in \S~\ref{S3.2}. We study polynomial pairs related to Theorem~\ref{polynoninj} \eqref{polynoninj2} in \S~\ref{S3.3}. Precisely, we consider polynomial pairs of the form $(z^r R(z)^k,z^r R(z^k))$, called \emph{Ritt moves} in this article. We show that $z^r R(z)^k$ and $z^r R(z^k)$ have the same multiplier spectrum over an arithmetic progression under specific conditions (Theorem~\ref{Rittmove}) and that they always have the same set of multipliers up to powers (Proposition~\ref{multidep}).

In \S~\ref{S4}, we present many problems and questions for future study. In \S~\ref{S4.1}, multiplier spectrum over arithmetic progressions and the related notion of stable multiplier spectrum are considered. In \S~\ref{S4.2}, we present several directions to generalize Theorem~\ref{polynoninj}.

\subsection*{Acknowledgement}
The author would like to express his great gratitude to his advisor, Professor Junyi Xie, for the constant encouragement. The author is supported by NSFC Grant (No.~12271007). The author would like to thank Fedor Pakovich and Valentin Huguin for helpful comments on the first version of this paper.

\section{Orbifolds and generalized Latt\`es maps}\label{SPak}
This section collects definitions and results concerning orbifolds and generalized Latt\`es maps necessary for our purposes. The theory of generalized Latt\`es maps (based on orbifolds) was primarily developed by Pakovich; see \cite{Pakgenla,Paksemiconj,Pakg01,PakgMos,Pakgfinite,Pakinvvar}. Here, we restrict our attention to orbifolds on the compact Riemann surface $\widehat{\C}=\P^1(\C)$, which suffices for our purposes.
\begin{Def}
	An \emph{orbifold} $\sO$ on $\widehat{\C}$ is defined by a \emph{ramification function} $\nu:\widehat{\C}\to\Z_{>0}$ such that $\{z\in\widehat{\C}\colon\nu(z)>1\}$ is finite. We consider only \emph{good orbifolds}, i.e., those satisfying $\#\{z\in\widehat{\C}\colon\nu(z)>1\}\neq1$, and if $\{z\in\widehat{\C}\colon\nu(z)>1\}=\{z_1,z_2\}$ has cardinality $2$, then $\nu(z_1)=\nu(z_2)$. For an orbifold $\sO$ on $\widehat{\C}$ with ramification function $\nu$, its \emph{Euler characteristic} is defined as
	$$\chi(\sO):=2+\sum_{z\in\widehat{\C}}\left(\frac{1}{\nu(z)}-1\right),$$
	the \emph{set of singular points} is $c(\sO)=\{z\in\widehat{\C}\colon\nu(z)>1\}$, and the \emph{signature} is the multiset $\nu(\sO)=\{\nu(z)\colon z\in c(\sO)\}$ of cardinality $\#c(\sO)$. The orbifold $\sO$ is called \emph{ramified} if its ramification function $\nu$ is not the constant function $1$ on $\widehat{\C}$.
\end{Def}
\begin{Def}
	Let $\sO_1$ and $\sO_2$ be two orbifolds on $\widehat{\C}$ with ramification functions $\nu_1$ and $\nu_2$, respectively, and let $f:\widehat{\C}\to\widehat{\C}$ be a non-constant rational map.
	\begin{enumerate}[(1)]
		\item We say $f:\sO_1\to\sO_2$ is a \emph{covering map} between orbifolds if for every $z\in\widehat{\C}$, we have
		$$\nu_2(f(z))=\nu_1(z)\cdot\deg_z(f),$$
		where $\deg_z(f)$ denotes the local degree of $f$ at $z$.
		\item We say $f:\sO_1\to\sO_2$ is a \emph{minimal holomorphic map} between orbifolds if for every $z\in\widehat{\C}$, we have
		$$\nu_2(f(z))=\nu_1(z)\cdot\gcd(\deg_z(f),\nu_2(f(z))).$$
	\end{enumerate}
\end{Def}
Clearly, every covering map between orbifolds is also a minimal holomorphic map.

A well-known equivalent description of Latt\`es maps is that a rational map $A:\widehat{\C}\to\widehat{\C}$ of degree $\geq2$ is a Latt\`es map if and only if there exists a ramified orbifold $\sO$ on $\widehat{\C}$ such that $A:\sO\to\sO$ is a covering map between orbifolds. This motivates the following definition:
\begin{Def}\label{genla}
A rational map $A:\widehat{\C}\to\widehat{\C}$ of degree $\geq2$ is called a \emph{generalized Latt\`es map} if there exists a ramified orbifold $\sO$ on $\widehat{\C}$ such that $A:\sO\to\sO$ is a minimal holomorphic map between orbifolds.
\end{Def}
From the definition, it is easy to check that all exceptional rational maps of degree $\geq2$ are generalized Latt\`es maps. Based on the proof of \cite[Theorem~7.2]{Pakgenla}, we obtain the following complete description of polynomials that are generalized Latt\`es maps:
\begin{Prop}\label{polygenla}
	Let $f\in\C[z]$ be a polynomial of degree at least $2$.
	\begin{enumerate}[(1)]
		\item \label{polygenla1} If $f$ is exceptional, then $f$ is a generalized Latt\`es map.
		\item \label{polygenla2} Suppose $f$ is non-exceptional. If $\sO$ is a ramified orbifold on $\widehat{\C}$ such that $f:\sO\to\sO$ is a minimal holomorphic map between orbifolds, then $\nu(\sO)=\{n,n\}$ for some integer $n\geq2$, and there exist $r\in\Z_{>0}$ with $\gcd(r,n)=1$ and a polynomial $R\in\C[z]\setminus\C$ such that $f$ is conjugate to $z^r R(z)^n$. In fact, $f$ is a generalized Latt\`es map if and only if there exist integers $r\geq1$ and $n\geq2$ with $\gcd(r,n)=1$, and a polynomial $R\in\C[z]$ such that $f$ is conjugate to $z^r R(z)^n$.
	\end{enumerate}
\end{Prop}
Using \cite[Theorem~1.1]{Pakgenla}, we easily obtain the following observation:
\begin{Lem}
	Let $f,g\in\C(z)\setminus\C$ be rational maps of degree $d\geq2$ such that $f\geq g$. If $f$ is a generalized Latt\`es map, then so is $g$.
\end{Lem}

One can associate two natural orbifolds with every rational map:
\begin{Def}
	Let $f:\widehat{\C}\to\widehat{\C}$ be a non-constant rational map.
	\begin{enumerate}[(i)]
		\item The orbifold $\sO_2^f$ is defined by the ramification function
		$$\nu_2(z)=\lcm\{\deg_w(f)\colon w\in f^{-1}(z)\},\quad z\in\widehat{\C}.$$
		\item The orbifold $\sO_1^f$ is defined by the ramification function
		$$\nu_1(z)=\nu_2(f(z))/\deg_z(f),\quad z\in\widehat{\C}.$$
	\end{enumerate}
\end{Def}

For any non-constant rational map $f:\widehat{\C}\to\widehat{\C}$, the map $f:\sO_1^f\to\sO_2^f$ is clearly a covering map between orbifolds.

Let $f\in\C(z)\setminus\C$ be a non-constant rational map. A rational map $f_0\in\C(z)$ is called a \emph{compositional right factor} of $f$ if there exists $f_1\in\C(z)$ such that $f=f_1\circ f_0$. Similarly, we define \emph{compositional left factors}. Pakovich \cite[\S~2]{Pakgenla} introduced the notion of good solutions for certain equations:
\begin{Def}\label{goodsol}
	Consider the functional equation
	\begin{equation}\label{fpgq}
		f\circ p=g\circ q,
	\end{equation}
	where $f,p,g,q\in\C(z)\setminus\C$. A solution $(f,p,g,q)$ of \eqref{fpgq} is called \emph{good} if $p$ and $q$ share no common compositional right factor of degree at least $2$, and the fiber product of $f$ and $g$ consists of a unique component.
\end{Def}
Pakovich \cite[Theorem~5.1]{Pakgenla} described good solutions of \eqref{fpgq} with $p(z)=g(z)=z^n$ ($n\geq2$) as follows:
\begin{Thm}\label{zngood}
	Let $n\geq2$ be an integer and let $A,F\in\C(z)$ be rational maps of degree at least $2$ such that $A\circ z^n=z^n\circ F$. Then the following are equivalent:
	\begin{enumerate}[(1)]
		\item The solution $(A,z^n,z^n,F)$ of \eqref{fpgq} is good.
		\item \label{zngood2} There exist $R\in\C(z)\setminus\C$ and $r\in\Z_{>0}$ with $\gcd(r,n)=1$ such that
		$$A(z)=z^r R(z)^n\quad\text{and}\quad F(z)=z^r R(z^n).$$
	\end{enumerate}
\end{Thm}
\begin{Rem}\label{zngoodpoly}
	When both $A$ and $F$ are polynomials, we may take $R$ in \eqref{zngood2} to be a polynomial as well.
\end{Rem}
Pakovich \cite{Pakgenla} showed that rational maps that are not generalized Latt\`es maps exhibit ``better'' behavior than generalized Latt\`es maps. The following theorem, a special case of \cite[Theorem~3.2]{Pakgenla}, illustrates this:
\begin{Thm}[Pakovich]\label{genla32}
	Let $A,B,\pi\in\C(z)$ be rational maps of degree at least $2$ such that $A\circ\pi=\pi\circ B$. Then one of the following holds:
	\begin{enumerate}
		\item $B\sim A$;
		\item $A$ is a generalized Latt\`es map and there exist $\psi,\pi_0,B_0\in\C(z)\setminus\C$ satisfying:
		\begin{itemize}
			\item $\pi=\pi_0\circ\psi$ and $\deg(\pi_0)\geq2$;
			\item $B_0\circ\psi=\psi\circ B$, $A\circ\pi_0=\pi_0\circ B_0$, and $B\sim B_0$;
			\item $(A,\pi_0,\pi_0,B_0)$ is a good solution of \eqref{fpgq};
			\item Both $A:\sO_2^{\pi_0}\to\sO_2^{\pi_0}$ and $B_0:\sO_1^{\pi_0}\to\sO_1^{\pi_0}$ are minimal holomorphic maps between orbifolds;
			\item There exists $s\in\Z_{>0}$ such that $\psi$ is a compositional right factor of $B^{\circ s}$ and a compositional left factor of $B_0^{\circ s}$.
		\end{itemize}
	\end{enumerate}
\end{Thm}

\section{The multiplier spectrum morphism is generically injective for polynomials}\label{S2}

\subsection{Pre-simple polynomials and the case $d\geq4$}\label{secpresimple}
A particular class of rational maps, namely simple rational maps, plays a key role in the proof of Theorem~\ref{ratgeninj} in \cite{JX23} (see \cite[Theorem~3.3]{JX23}). This notion was introduced by Pakovich \cite{Paksimple}. A rational map $f\in\C(z)$ of degree $d\geq2$ is \emph{simple} if it has exactly $2d-2$ distinct critical values in $\P^1(\C)$. By the Riemann–Hurwitz formula, every $f\in\C(z)$ of degree $d$ has exactly $2d-2$ critical points counted with multiplicity. Hence, a simple rational map is a rational map with all critical points of multiplicity $1$ such that no two distinct critical points map to the same value. 

For rational maps, simplicity is a natural notion since simple rational maps of degree $d\geq2$ form a Zariski open and dense subset of $\Rat_d(\C)$. However, every polynomial $f\in\C[z]$ of degree $d\geq2$ has a critical point of multiplicity $d-1$ at $\infty$. Thus, $f$ is simple if and only if $d=2$. The notion of simplicity is thus ill-suited for polynomials. To overcome this, we introduce a modified notion for polynomials:
\begin{Def}
	A polynomial $f(z)\in\C[z]$ of degree $d\geq2$ is called \emph{pre-simple} if it has exactly $d-1$ distinct critical values in $\C$.
\end{Def}
Note that $d-1$ is the maximum possible number of distinct finite critical values for a polynomial of degree $d$. We show that most statements in \cite{Paksimple} concerning simple rational maps admit analogues for pre-simple polynomials. In particular, we obtain a polynomial version (Theorem~\ref{poly33}) of \cite[Theorem~3.3]{JX23}, which yields a proof of Theorem~\ref{polygeninj} for $d\geq4$. Most proofs in this subsection are minor modifications of those in \cite{Paksimple}; we provide complete details for the reader's convenience.

We first recall some notions about decomposing rational maps. Let $f\in\C(z)$ be a rational map of degree $d\geq2$. We say that $f$ is \emph{indecomposable} if $f=f_1\circ f_2$ in $\C(z)$ implies $\deg(f_1)=1$ or $\deg(f_2)=1$. A \emph{decomposition} of $f$ is a representation $f=f_1\circ\cdots\circ f_r$, where each $f_i\in\C(z)$ has degree $\geq2$. Two decompositions $f=f_1\circ\cdots\circ f_r$ and $f=g_1\circ\cdots\circ g_l$ of $f$ are called \emph{equivalent} if either $l=r=1$ and $f_1=g_1$, or $l=r\geq2$ and there exist $\nu_1,\dots,\nu_{r-1}\in\PGL_2(\C)$ such that $f_1=g_1\circ\nu_1$, $f_j=\nu_{j-1}^{-1}\circ g_j\circ\nu_j$ for $1<j<r$, and $f_r=\nu_{r-1}^{-1}\circ g_r$. A decomposition $f=f_1\circ\cdots\circ f_r$ of $f$ is called \emph{complete} if $f_j$ is indecomposable for every $1\leq j\leq r$. It is easy to see that any rational map of degree $d\geq2$ admits a complete decomposition by an induction on $d$.
\begin{Prop}\label{Mon}
	A pre-simple polynomial $f$ of degree $d\geq2$ is indecomposable (as a rational map) and satisfies $\Mon(f)\cong S_d$, where $S_d$ is the permutation group on $\{1,2,\dots,d\}$ and $\Mon(f)$ is the monodromy group of $f$.
\end{Prop}
\begin{proof}
	Let $f\in\C[z]$ be a pre-simple polynomial of degree $d\ge2$. Suppose $f$ is decomposable, so $f=f_1\circ f_2$ with $f_1,f_2\in\C(z)$ of degrees $m_1,m_2\geq2$, respectively. Then $d=m_1 m_2$. Since $\infty$ is a totally invariant fixed point of $f=f_1\circ f_2$, the point $a=f_2(\infty)$ satisfies $f_1^{-1}(\infty)=\{a\}$ and $f_2^{-1}(a)=\{\infty\}$. After composing with a suitable M\"obius transformation, we may assume $a=\infty$. Then both $f_1$ and $f_2$ are polynomials. The pre-simplicity of $f$ implies that the number $N(f)$ of critical values of $f$ in $\C$ equals $(2d-2)-(d-1)=d-1$. On the other hand, from $f^\prime=(f_1^\prime\circ f_2)\cdot f_2^\prime$ we see that
	$$m_1m_2-1=d-1=N(f)\leq N(f_1)+N(f_2)\leq m_1-1+m_2-1.$$
	Equivalently, $(m_1-1)(m_2-1)\leq0$, contradicting the assumption that $m_1,m_2\geq2$. Hence, $f$ is indecomposable.
	
	Since $f$ is indecomposable, its monodromy group $\Mon(f)$ is primitive. Fix a critical value $c$ of $f$ in $\C$. Since $f$ is pre-simple, the fiber $f^{-1}(c)$ contains exactly one critical point with local degree $2$ and $d-2$ unramified points; hence, the corresponding permutation in $\Mon(f)$ is a transposition. By \cite[Theorem~13.3]{PermGp}, $\Mon(f)$ must be the full symmetric group, so $\Mon(f)\cong S_d$.
\end{proof}

We now introduce algebraic curves associated with polynomial pairs, following \cite[\S~2]{Paksimple}.
\begin{Def}
	Let $F(z),H(z)\in\C[z]$ be polynomials of degrees $m,n\geq1$, respectively. Let $h_{F,H}$ be the algebraic curve (in $\A^2$) defined by $F(x)-H(y)=0$.
\end{Def}
The genus of the curve $h_{F,H}$ can be computed explicitly as follows (see \cite{Fried} or \cite{Pak11}):
\begin{Prop}\label{genus}
	Let $F(z),H(z)\in\C[z]$ be polynomials of degrees $m,n\geq1$, respectively, and let $S=\{z_1,\dots,z_r\}$ be the union of the sets of critical values of $F$ and $H$ in $\C$. For $1\leq i\leq r$, let $(a_{i,1},\dots,a_{i,p_i})$ and $(b_{i,1},\dots,b_{i,q_i})$ denote the multiplicities of $F$ and $H$ at the points in $F^{-1}(z_i)$ and $H^{-1}(z_i)$, respectively. Then the genus of $h_{F,H}$ is determined by
	$$2-2g(h_{F,H})=\gcd(m,n)-(r-1)mn+\sum_{i=1}^r\sum_{j_1=1}^{p_i}\sum_{j_2=1}^{q_i}\gcd(a_{i,j_1},b_{i,j_2}).$$
\end{Prop}
The following theorem is an analogue of \cite[Theorem~2.3]{Paksimple}.
\begin{Thm}\label{hFH}
	Let $m\geq4$ and $n\geq2$ be integers such that $(m,n)$ is not $(4,2)$ or $(4,4)$. Let $F\in\C[z]$ be a pre-simple polynomial of degree $m$ and $H\in\C[z]$ a polynomial of degree $n$. Assume the curve $h_{F,H}$ is irreducible. Then $g(h_{F,H})>0$. In particular, there exist no non-constant rational maps $X(z),Y(z)\in\C(z)\setminus\C$ such that $F\circ X=H\circ Y$ in $\C(z)$.
\end{Thm}
\begin{proof}
	We use the notation from Proposition~\ref{genus}. Let $\critV(F)$ and $\critV(H)$ denote the sets of critical values of $F$ and $H$ in $\C$, respectively.
	
	For $i\in\{1,\dots,r\}$ such that $z_i\notin\critV(F)$, we have $p_i=m$ and $(a_{i,1},\dots,a_{i,m})=(1,\dots,1)$, so
	$$\sum_{j_1=1}^{p_i}\sum_{j_2=1}^{q_i}\gcd(a_{i,j_1},b_{i,j_2})=\sum_{j_1=1}^{p_i}q_i=p_i q_i=m q_i.$$
	
	For $i\in\{1,\dots,r\}$ such that $z_i\in\critV(F)$, since $F$ is pre-simple, we have $p_i=m-1$ and $\deg_z(F)\in\{1,2\}$ for every $z\in\C$. We may assume $$(a_{i,1},\dots,a_{i,m-2},a_{i,m-1})=(1,\dots,1,2).$$
	Let $l_i$ be the number of indices $j_2\in\{1,\dots,q_i\}$ such that $b_{i,j_2}$ is even. Then
	$$\sum_{j_1=1}^{p_i}\sum_{j_2=1}^{q_i}\gcd(a_{i,j_1},b_{i,j_2})=(m-2)q_i+q_i+l_i=mq_i+(l_i-q_i).$$
	
	By the Riemann–Hurwitz formula,
	$$(2n-2)-(n-1)=\sum_{z\in\C}(\deg_z(H)-1)=\sum_{i=1}^r\sum_{j_2=1}^{q_i}(b_{i,j_2}-1)=rn-\sum_{i=1}^r q_i,$$
	so $\sum_{i=1}^r q_i=(r-1)n+1$. Thus,
	$$\sum_{i=1}^r\sum_{j_1=1}^{p_i}\sum_{j_2=1}^{q_i}\gcd(a_{i,j_1},b_{i,j_2})=\sum_{i=1}^r mq_i+A=(r-1)mn+m+A,$$
	where $A:=\sum_{1\leq i\leq r,\,z_i\in\critV(F)}(l_i-q_i)$. By Proposition~\ref{genus}, the genus is
	$$g(h_{F,H})=\frac{1}{2}\cdot(2-m-\gcd(m,n)-A).$$
	Hence, $g(h_{F,H})=0$ if and only if $-A=m+\gcd(m,n)-2$. Note that $-A$ equals the number of $z\in\C$ such that $H(z)\in\critV(F)$ and $\deg_z(H)$ is odd.
	
	For any finite subset $S\subset\C$, we have
	$$n-1\geq\sum_{z\in H^{-1}(S)}(\deg_z(H)-1)=n(\# S)-\# H^{-1}(S);$$
	hence $\# H^{-1}(S)\geq n(\# S)-(n-1)=n(-1+\# S)+1$, with equality if and only if $\critV(H)\subseteq S$. In particular, $\# H^{-1}(\critV(F))\geq n(m-2)+1$, with equality if and only if $\critV(H)\subseteq\critV(F)$.
	
	Suppose $g(h_{F,H})=0$. Then $-A=m+\gcd(m,n)-2$, which implies
	\begin{align*}
		&\# H^{-1}(\critV(F))\\
		\leq&m+\gcd(m,n)-2+\left\lfloor\frac{n(m-1)-m-\gcd(m,n)+2}{2}\right\rfloor\\
		=&\frac{1}{2}\left((m-1)(n+1)+\gcd(m,n)-1\right).
	\end{align*}
	Thus,
	$$(m-1)(n+1)+\gcd(m,n)-1\geq2\#H^{-1}(\critV(F))\geq2n(m-2)+2,$$
	or equivalently,
	$$mn+4\leq3n+m+\gcd(m,n).$$
	Then $mn+4\leq3n+m+n$, so $(m-4)(n-1)\leq0$. As $n\geq2$ and $m\geq4$, we must have $m=4$. Then $n\neq2,4$ by assumption. Now $mn+4\leq3n+m+\gcd(m,n)$ becomes $n\leq\gcd(4,n)$, contradicting the assumption that $n\neq2,4$. Therefore, $g(h_{F,H})>0$.
	
	Now suppose there exist $X(z),Y(z)\in\C(z)\setminus\C$ such that $F\circ X=H\circ Y$. Let $C$ be the Zariski closure of $h_{F,H}$ in $\P^1\times\P^1$ , which is an irreducible projective curve. Since genus is a birational invariant, $g(C)=g(h_{F,H})>0$. The morphism $(X,Y):\P^1\to\P^1\times\P^1,z\mapsto(X(z),Y(z))$ has image contained in $C$. So it induces a non-constant morphism $\phi:\P^1\to C$ between irreducible projective curves. However, as $g(\P^1)=0$ and $g(C)>0$, no non-constant morphism $\P^1\to C$ exists by the Riemann–Hurwitz formula, which is a contradiction.
\end{proof}

For a non-constant rational map $f\in\C(z)$, define the groups
$$\Sigma(f):=\{\sigma\in\PGL_2(\C)\colon f\circ\sigma=f\}\quad\text{and}\quad\Sigma_\infty(f):=\bigcup_{k\geq1}\Sigma(f^{\circ k}).$$
Clearly, $\Sigma(f)\subseteq\Sigma_\infty(f)$, and this inclusion is preserved under conjugacy in the sense that for $\tau\in\PGL_2(\C)$,
\begin{equation}\label{Sigconj}
	\Sigma(\tau\circ f\circ\tau^{-1})=\tau\circ\Sigma(f)\circ\tau^{-1}\subseteq\tau\circ\Sigma_\infty(f)\circ\tau^{-1}=\Sigma_\infty(\tau\circ f\circ\tau^{-1}).
\end{equation}
For a polynomial $f\in\C[z]\setminus\C$, the totally invariant fixed point $\infty$ of $f^{\circ k}$ ($k\geq1$) implies that
\begin{equation}\label{Siginfpoly}
	\Sigma_\infty(f)\subseteq\Aff(\C).
\end{equation}

\begin{Lem}\label{presimSig1}
	Let $f\in\C[z]$ be a pre-simple polynomial of degree $m\geq3$. Then $\Sigma(f)=1$.
\end{Lem}
\begin{proof}
	Since $f$ is pre-simple of degree $m\geq3$, $f$ is not conjugate to $z^m$. After conjugacy, we may assume that $f$ is of the form
	$$f(z)=z^m+a_{m-2}z^{m-2}+\cdots+a_0,$$
	where $a_{m-2},\dots,a_0\in\C$. Let $\tau\in\Sigma(f)$ and write $\tau(z)=\la z+c$ with $\la\in\C^\times$ and $c\in\C$ by \eqref{Siginfpoly}. Comparing the coefficients of $z^m$ and $z^{m-1}$, we see that $c=0$ and $\la^m=1$. Assume that $\la\neq1$. Let $l\geq2$ be the order of the root of unity $\la$ in $\C^\times$. Then $l\mid m$. As $f(\la z)=f(z)$, we have $f(z)=g(z^l)$ for a unique polynomial $g\in\C[z]$ of degree $m/l$. Since $f$ is pre-simple, its derivative $f^\prime(z)=lz^{l-1}g^{\prime}(z^l)$ has no multiple roots. In particular, $0$ is not a multiple root of $f^\prime$; hence $l=2$ (otherwise $z^2\mid f^\prime(z)$). Note that $f^\prime(z)=2z g^{\prime}(z^2)\in\C[z]$ is a polynomial of degree $m-1\geq2$ that is an odd function on $\C$ without multiple roots. Hence there exists $\omega_0\in\C^\times$ such that $f^\prime(\omega_0)=f^\prime(-\omega_0)=0$. However, the two distinct critical points $\omega_0\neq-\omega_0$ of $f$ are mapped by $f$ to the same value $g(\omega_0^2)$, contradicting the assumption that $f$ is pre-simple. This completes the proof.
\end{proof}
The following theorem (see also Remark~\ref{hFrmk}) is an analogue of \cite[Theorem~2.4]{Paksimple}.
\begin{Thm}\label{hF}
	Let $F\in\C[z]$ be a pre-simple polynomial of degree $m\geq3$. Then the equality $F\circ X=F\circ Y$ in $\C[z]$ with $X,Y\in\C[z]\setminus\C$ implies $X=Y$.
\end{Thm}
\begin{proof}
	Let $X,Y\in\C[z]$ be non-constant polynomials such that $F\circ X=F\circ Y$. Set $n=\deg(X)=\deg(Y)\in\Z_{>0}$.
	
	By \cite[Theorem~3.54]{FG22}, there exist polynomials $$\tilde{F},\mu_1,\mu_2,G,\sigma_1,\sigma_2\in\C[z]\setminus\C$$
	with
	$$\deg(\tilde{F})=m\quad\text{and}\quad\deg(G)=n$$
	such that
	$$F=\tilde{F}\circ\mu_1=\tilde{F}\circ\mu_2,\quad\sigma_1\circ G=X,\quad\sigma_2\circ G=Y,\quad\text{and}\quad\mu_1\circ\sigma_1=\mu_2\circ\sigma_2.$$
	Note that $\mu_1$, $\mu_2$, $\sigma_1$, and $\sigma_2$ are polynomials of degree one. Then $F\circ\mu_1^{-1}\circ\mu_2=F$ and we deduce that $\mu_1=\mu_2$ by Lemma~\ref{presimSig1}. Thus $\sigma_1=\mu_1^{-1}\circ\mu_2\circ\sigma_2=\sigma_2$. So $X=\sigma_1\circ G=\sigma_2\circ G=Y$. The proof is complete.
\end{proof}
\begin{Rem}\label{hFrmk}
	The presented proof of Theorem~\ref{hF} is based on Ritt's theory of polynomial decompositions.
	
	If the pre-simple polynomial $F\in\C[z]$ has degree $m\geq4$, then we can show that the algebraic curve $h_F$ in $\A^2$ given by
	$$\frac{F(x)-F(y)}{x-y}=0$$
	is irreducible and has genus $g(h_F)>0$, using the argument in the proof of \cite[Theorem~2.4]{Paksimple}. Hence for pre-simple $F\in\C[z]$ of degree $m\geq4$, we conclude that the equality $F\circ X=F\circ Y$ in $\C(z)$ with $X,Y\in\C(z)\setminus\C$ implies $X=Y$.
\end{Rem}
Using Theorem~\ref{hF}, we obtain the following generalization of Lemma~\ref{presimSiginf1}:
\begin{Cor}\label{presimSiginf1}
	Let $f\in\C[z]$ be a pre-simple polynomial of degree $m\geq3$. Then $\Sigma(f)=\Sigma_\infty(f)=1$.
\end{Cor}
\begin{proof}
	Let $\tau\in\Sigma(f^{\circ n})$ for some integer $n\in\Z_{\geq0}$. Then $\tau$ is a polynomial of degree one by \eqref{Siginfpoly}. Applying Theorem~\ref{hF} recursively to the equation $f^{\circ n}\circ \tau=f^{\circ n}$, we see that $\tau(z)=z$. Therefore, $\Sigma_\infty(f)=1$ is the trivial group and so is $\Sigma(f)$.
\end{proof}
\begin{Rem}
	Note that an iterate $f^{\circ n}$ of a pre-simple polynomial $f\in\C[z]$ may not be pre-simple when $n\geq2$; so Corollary~\ref{presimSiginf1} is not a direct result of Lemma~\ref{presimSig1}.
\end{Rem}

By Ritt's theory of polynomial decompositions, we see that self-compositions of an indecomposable polynomial behave well under decomposition:
\begin{Prop}\label{Fldecomp}
	Let $F\in\C[z]$ be a polynomial of degree $m\geq2$. Suppose that $F$ is indecomposable (for example, when $F$ is pre-simple (Proposition~\ref{Mon}) or $m$ is prime). Then for every $l\in\Z_{>0}$, every complete decomposition of $F^{\circ l}$ is equivalent to the complete decomposition $F^{\circ l}=F\circ\cdots\circ F$.
\end{Prop}
\begin{proof}
	The case $l=1$ is trivial. Assume $l\geq2$ and let $F^{\circ l}=F_1\circ\cdots\circ F_r$ be a complete decomposition of $F^{\circ l}$. Since $F^{\circ l}=F\circ\cdots\circ F$ is a complete decomposition of $F^{\circ l}$ of length $l$, we have $r=l$ by \cite[Theorem~3.35]{FG22}. Up to equivalent (complete) decompositions, the totally invariant fixed point $\infty$ implies that we may assume that $F_1,\dots,F_l$ are all polynomials.
	
	Note that $n:=\gcd(\deg(F_1),m)\geq2$ because $\deg(F_1)\mid m^l=\deg(F^{\circ l})$. By \cite[Theorem~3.54]{FG22}, there exist polynomials $G,\sigma,\mu\in\C[z]\setminus\C$ such that
	$$F=G\circ\sigma,\quad F_1=G\circ\mu,\quad\text{and}\quad\sigma\circ F^{\circ (l-1)}=\mu\circ F_2\circ\cdots\circ F_l$$
	with $\deg(G)=n>1$.
	Since both $F$ and $F_1$ are indecomposable, we conclude that
	$$n=\deg(G)=\deg(F_1)=\deg(F)=m\quad\text{and}\quad\deg(\sigma)=\deg(\mu)=1.$$
	Set $\tau=\sigma^{-1}\circ\mu\in\Aff(\C)$. Then we obtain
	$$F_1=F\circ\tau\quad\text{and}\quad F^{\circ (l-1)}=(\tau\circ F_2)\circ F_3\circ\cdots\circ F_l.$$
	An inductive argument then finishes the proof.
\end{proof}
\begin{Cor}\label{Corpresimdecom}
	Let $F\in\C[z]$ be a polynomial of degree $m\geq2$. Suppose that $F$ is indecomposable (for example, when $F$ is pre-simple or $m$ is prime). Let $G_1,\dots,G_r\in\C(z)$ be rational maps of degree $\geq2$ such that $F^{\circ l}=G_1\circ\cdots\circ G_r$, where $r,l\in\Z_{>0}$. Then there exist $\nu_1,\dots,\nu_{r-1}\in\PGL_2(\C)$ such that
	$$G_1=F^{\circ s_1}\circ\nu_1,\quad G_j=\nu_{j-1}^{-1}\circ F^{\circ s_j}\circ\nu_j\ (1<j<r),\quad G_r=\nu_{r-1}^{-1}\circ F^{\circ s_r},$$
	where $s_i:=\log_m\deg(G_i)$ is a positive integer for $1\leq i\leq r$. If $G_1,\dots,G_r$ are polynomials, then we can choose $\nu_1,\dots,\nu_{r-1}\in\Aff(\C)$.
\end{Cor}
\begin{proof}
For every $1\leq i\leq r$, choose a complete decomposition $G_i=G_{i,1}\circ\cdots\circ G_{i,k_i}$ of $G_i$. By connecting these decompositions, we obtain a complete decomposition 
\begin{equation}\label{connectdecomp}
	F^{\circ l}=G_{1,1}\circ\cdots\circ G_{1,k_1}\circ\cdots\circ G_{r,1}\circ\cdots\circ G_{r,k_r}
\end{equation}
of $F^{\circ l}$. The conclusion then follows by applying Proposition~\ref{Fldecomp} to the decomposition \eqref{connectdecomp}.
\end{proof}

Let $F\in\C(z)$ be a rational map of degree $m\geq2$. Apart from $\Sigma(F)$ and $\Sigma_\infty(F)$, we recall some other (semi)groups (under composition) associated with $F$ of degree $m\geq2$, following \cite{Paksimple}.

Let
$$C(F):=\{g\in\C(z)\setminus\C\mid F\circ g=g\circ F\}$$
be the semigroup of non-constant rational maps commuting with $F$. Let
$$C_\infty(F):=\bigcup_{l=1}^\infty C(F^{\circ l})$$
be the semigroup of non-constant rational maps commuting with some iterate of $F$. Define
$$\Aut(F):=C(F)\cap\PGL_2(\C)\quad\text{and}\quad\Aut_\infty(F):=C_\infty(F)\cap\PGL_2(\C).$$
Both $\Aut(F)$ and $\Aut_\infty(F)$ are groups. Note that the semigroup $\left\langle\Aut_\infty(F),F\right\rangle$ generated by $\Aut_\infty(F)\cup\{F\}$ is contained in $C_\infty(F)$.

For $g\in\C(z)$ of degree $\geq2$, let $\mu_g$ be the measure of maximal entropy of $g$ on $\P^1(\C)$ (see \cite{Lyubich83,Mane83}). Define
$$E_0(F):=\{\sigma\in\PGL_2(\C)\mid\sigma_*\mu_F=\mu_F\}$$
and
$$E(F):=E_0(F)\cup\{g\in\C(z)\mid\deg(g)\geq2,\mu_g=\mu_F\}.$$
Then $E_0(F)$ is a group, and $E(F)$ is a semigroup (see \cite{Esemigp}).

Define $G(F)$ as the subgroup of $\PGL_2(\C)$ given by
$$\{\sigma\in\PGL_2(\C)\mid\exists\tau\in\PGL_2(\C),\,F\circ\sigma=\tau\circ F\}.$$
Note that for every $\sigma\in G(F)$, there is a unique $\tau_\sigma\in\PGL_2(\C)$ such that
\begin{equation}\label{tausig}
	F\circ\sigma=\tau_\sigma\circ F.
\end{equation}
The group $G(F)$ is called the \emph{extended symmetry group} of $F$.
Define $G_0(F)$ as the maximal subgroup of $\PGL_2(\C)$ such that for every $\sigma\in G_0(F)$, there exists $\tau\in G_0(F)$ with $F\circ\sigma=\tau\circ F$. From the definitions, it is clear that
$$\left\langle\Aut(F)\cup\Sigma(F)\right\rangle\subseteq G_0(F)\subseteq G(F).$$

Suppose that $F$ is a polynomial. Then these semigroups associated with $F$ are expected to ``relate to'' polynomials. The inclusion $\Sigma_\infty(F)\subseteq\Aff(\C)$ (see \eqref{Siginfpoly}) and the following two lemmas illustrate such an idea.
\begin{Lem}\label{Cinfpoly}
	Let $F\in\C[z]$ be a polynomial of degree $m\geq2$ that is not conjugate to $z^m$. Then $C_\infty(F)\subseteq\C[z]\setminus\C$.
\end{Lem}
\begin{proof}
	We work over the algebraically closed field $\C$ of characteristic zero. For a polynomial $F\in\C[z]$ of degree $m\geq2$, it is well-known that $F$ is conjugate to $z^m$ if and only if for some (hence every) $k\in\Z_{>0}$, the iterate $F^{\circ k}$ is conjugate to $z^{m^k}$; see \cite{milnor2006lattes,Zdunik90}. Then the conclusion follows immediately from \cite[Theorem~6.80]{Silverman2007}.
\end{proof}

\begin{Lem}\label{Epoly}
	Let $F\in\C[z]$ be a polynomial of degree $m\geq2$ that is non-exceptional. Then $E(F)\subseteq\C[z]\setminus\C$.
\end{Lem}
\begin{proof}
	We first show that $E(F)\setminus E_0(F)\subseteq\C[z]\setminus\C$. Let $g\in E(F)\setminus E_0(F)$. Since $F$ is non-exceptional, by \cite{Levin}, there exist integers $k_1,l>0$ and $k_2\geq0$ such that $F^{\circ k_1}=F^{\circ k_2}\circ g^{\circ l}$. Note that $(g^{\circ l})^{-1}(\infty)=\{\infty\}$, hence $g^{\circ l}$ is a polynomial. Since $\mu_{g^{\circ l}}=\mu_g=\mu_F$ and $F$ is non-exceptional, $g^{\circ l}$ is also non-exceptional by \cite[Corollary~2.17]{FG22}. In particular, $g^{\circ l}$ is not conjugate to $z^{n^l}$, where $n:=\deg(g)\geq2$. By Lemma~\ref{Cinfpoly}, $g\in C(g^{\circ l})$ is a polynomial.
	
	Let $\sigma\in E_0(F)$. Then the rational map $\tilde{F}=\sigma\circ F\circ\sigma^{-1}\in\C(z)$ satisfies $\mu_{\tilde{F}}=\sigma_*\mu_F=\mu_F$, so $\tilde{F}\in E(F)\setminus E_0(F)\subseteq\C[z]$ is a polynomial. As $\infty$ is the unique totally invariant fixed point of the non-exceptional polynomial $F$ (and of $\tilde{F}$), we have $\sigma(\infty)=\infty$. We conclude that $\sigma\in\Aff(\C)$.
\end{proof}
\begin{Rem}
	If $F\in\C[z]$ (of degree $m\geq2$) is exceptional, then $E(F)\nsubseteq\C[z]\setminus\C$. See \cite[\S~7]{Milnor}.
\end{Rem}

\begin{Lem}\label{gam}
	Let $F\in\C[z]$ be a pre-simple polynomial of degree $m\geq3$. Define $\gamma=\gamma_F:G(F)\to\PGL_2(\C),\sigma\mapsto\tau_\sigma$ where $\tau_\sigma$ is as in \eqref{tausig}. Then $G(F)$ is finite, and $\gamma$ restricts to a group automorphism of $G_0(F)$.
\end{Lem}
\begin{proof}
	By \cite[Theorem 4.2]{Pakgfinite}, the group $G(F)$ is finite of order bounded in terms of $m=\deg(F)$, unless $(\alpha\circ F\circ\beta)(z)=z^m$ for some $\alpha,\beta\in\PGL_2(\C)$. Since $F$ is pre-simple, for any M\"obius transformations $\alpha$ and $\beta$, the map $\alpha\circ F\circ\beta$ has exactly $m$ critical values in $\widehat{\C}$. But $z^m$ has only $2(<m)$ critical values in $\widehat{\C}$. Hence, $G(F)$ is finite, and so is its subgroup $G_0(F)$. By definition, $\gamma(G_0(F))\subseteq G_0(F)$. By Theorem~\ref{hF}, $\Sigma(F)=1$, which implies that $\gamma$ is injective. Therefore, the injective self-map $\gamma:G_0(F)\to G_0(F)$ on the finite set $G_0(F)$ is a bijection. It is easy to check that $\gamma$ is a group homomorphism, hence an automorphism.
\end{proof}
\begin{Cor}\label{G0Auts}
	Let $F\in\C[z]$ be a pre-simple polynomial of degree $m\geq3$. Let $s=\#\Aut(G_0(F))\in\Z_{>0}$, where $\Aut(G_0(F))$ is the automorphism group of the finite group $G_0(F)$. Then $G_0(F)\subseteq\Aut(F^{\circ s})$.
\end{Cor}
\begin{proof}
	By Lemma~\ref{gam}, $\gamma:G_0(F)\to G_0(F)$ is an automorphism of the finite group $G_0(F)$. Since $s=\#\Aut(G_0(F))$, $\gamma^{\circ s}=\Id_{G_0(F)}$ in $\Aut(G_0(F))$. For every $\sigma\in G_0(F)$, we have
	$$F^{\circ s}\circ\sigma=\gamma^{\circ s}(\sigma)\circ F^{\circ s}=\sigma\circ F^{\circ s},$$
	hence $G_0(F)\subseteq\Aut(F^{\circ s})$.
\end{proof}
By definition, it is easy to see that pre-simple polynomials of degree $\geq4$ are non-exceptional. 
\begin{Lem}\label{presimnonexcep}
	Every pre-simple polynomial $F\in\C[z]$ of degree $m\geq4$ is non-exceptional.
\end{Lem}
\begin{proof}
	Suppose $F$ is a pre-simple polynomial of degree $m\geq4$ that is exceptional. Since $F$ is a polynomial, it is not a Latt\`es map. Hence $F$ is of monomial type; so $F$ is affinely conjugate to $z^{\pm m}$ or $\pm T_m$, where $T_m\in\C[z]$ is the Chebyshev polynomial of degree $m$. As $F$ is a polynomial, it cannot be conjugate to $z^{-m}$. Hence, $F$ is affinely conjugate to $z^m$ or $\pm T_m$. Clearly, $z^m$ is not pre-simple for $m\geq3$. It is well-known that $\pm T_m$ are polynomial solutions to the differential equation $(4-z^2)G^{\prime}(z)^2=m^2(4-G(z)^2)$ (see \cite[Lemma~6.10]{Silverman2007}); so the only choices of finite critical values of $\pm T_m$ are $\pm2$. Since $2<m-1$, $\pm T_m$ is not pre-simple. Hence we get a contradiction to the fact that pre-simplicity is preserved by affine conjugacy.
\end{proof}
\begin{Rem}
	Every polynomial of degree $2$ is automatically pre-simple. Note that the Chebyshev polynomial $T_3(z)=z^3-3z$ of degree $3$ is pre-simple.
\end{Rem}
The following theorem describes the relations between the semigroups introduced above for a pre-simple polynomial of degree $\geq4$, analogous to \cite[Theorem~1.2]{Paksimple}.
\begin{Thm}\label{presimrel}
	Let $F\in\C[z]$ be a pre-simple polynomial of degree $m\geq4$. Set $s=\#\Aut(G_0(F))$, which is a positive integer. Then:
	\begin{itemize}
		\item $G_0(F)=\Aut(F^{\circ s})=\Aut_{\infty}(F)=E_0(F)\subseteq\Aff(\C)$;
		\item $C_\infty(F)=E(F)=\left\langle\Aut_{\infty}(F),F\right\rangle\subseteq\C[z]$.
	\end{itemize}
\end{Thm}
\begin{proof}
	By Lemma~\ref{presimnonexcep}, $F$ is non-exceptional. Then we have $E_0(F)\subseteq\Aff(\C)$ and $E(F)\subseteq\C[z]$ by Lemma~\ref{Epoly}.
	
	Note that $G_0(F)\subseteq\Aut(F^{\circ s})\subseteq\Aut_{\infty}(F)\subseteq E_0(F)$, where the first inclusion is by Corollary~\ref{G0Auts}, the second by definition, and the third by \cite[(29)]{Paksimple}.
	
	We show that $E_0(F)\subseteq G_0(F)$. Let $\sigma\in E_0(F)\subseteq\Aff(\C)$. Then $F\circ\sigma\in E(F)$. By \cite{Levin}, there exist integers $k_1,l>0$ and $k_2\geq0$ with $k_1=k_2+l$ such that $F^{\circ k_1}=F^{\circ k_2}\circ(F\circ\sigma)^{\circ l}$. Applying Theorem~\ref{hF} recursively, we get $F^{\circ l}=F^{\circ (k_1-k_2)}=(F\circ\sigma)^{\circ l}$. By Theorem~\ref{Fldecomp}, there exist $\mu_1,\dots,\mu_{l-1}\in\Aff(\C)$ such that
	$$F\circ\sigma=F\circ\mu_1,\quad F\circ\sigma=\mu_{i-1}^{-1}\circ F\circ\mu_i\ (1<i<l),\quad F\circ\sigma=\mu_{l-1}^{-1}\circ F.$$
	Set $\mu_0(z)=z$ when $l=1$. Then $\tau:=\mu_{l-1}^{-1}$ satisfies $F\circ\sigma=\tau\circ F$. We claim $\tau\in E_0(F)$. For $l=1$, this is clear. For $l\geq2$, note that
	$$F^{\circ l}=(F\circ\sigma)^{\circ l}=(\tau\circ F)^{\circ l}.$$
	By \cite[Lemma~3.5]{Paksimple}, $\tau\in\Aut(F^{\circ l})\subseteq\Aut_{\infty}(F)\subseteq E_0(F)$. By the maximality in the definition of $G_0(F)$, we obtain $E_0(F)\subseteq G_0(F)$.
	
	Thus, $G_0(F)=\Aut(F^{\circ s})=\Aut_{\infty}(F)=E_0(F)\subseteq\Aff(\C)$.
	
	\medskip
	
	By \cite[(28)]{Paksimple}, $C_\infty(F)\subseteq E(F)$. Note that
	$$C_\infty(F)\cap\PGL_2(\C)=\Aut_{\infty}(F)=E_0(F)=E(F)\cap\PGL_2(\C).$$
	To show $C_\infty(F)=E(F)$, it suffices to prove $E(F)\setminus E_0(F)\subseteq C_\infty(F)\setminus\PGL_2(\C)$. Let $g\in E(F)\setminus E_0(F)$. Then $\deg(g)\geq2$ and $\mu_g=\mu_F$. By \cite{Levin}, there exist integers $k_1,l>0$ and $k_2\geq0$ such that $F^{\circ k_1}=F^{\circ k_2}\circ g^{\circ l}$. Clearly, $k_1>k_2$. Applying Theorem~\ref{hF} recursively, we get $F^{\circ (k_1-k_2)}=g^{\circ l}$. Then $g$ commutes with $F^{\circ (k_1-k_2)}=g^{\circ l}$, so $g\in C_\infty(F)\setminus\PGL_2(\C)$. Hence, $C_\infty(F)=E(F)$.
	
	Clearly, $\left\langle\Aut_{\infty}(F),F\right\rangle\subseteq C_\infty(F)$. For M\"obius transformations,
	$$\left\langle\Aut_{\infty}(F),F\right\rangle\cap\PGL_2(\C)=\Aut_{\infty}(F)=C_\infty(F)\cap\PGL_2(\C).$$
	To show $\left\langle\Aut_{\infty}(F),F\right\rangle=C_\infty(F)$, it suffices to prove
	$$C_\infty(F)\setminus\PGL_2(\C)\subseteq\left\langle\Aut_{\infty}(F),F\right\rangle\setminus\PGL_2(\C).$$
	Let $G\in C_\infty(F)\setminus\PGL_2(\C)$. By \cite{RittPerm} and Lemma~\ref{presimnonexcep}, there exist integers $k,l>0$ such that $F^{\circ k}=G^{\circ l}$. By Corollary~\ref{Corpresimdecom}, there exists $\mu\in\PGL_2(\C)$ such that $G=\mu\circ F^{\circ s}$, where $s=k/l\in\Z_{>0}$ (take $\mu(z)=z$ when $l=1$). Then
	$$(F^{\circ s})^{\circ l}=F^{\circ k}=G^{\circ l}=(\mu\circ F^{\circ s})^{\circ l}.$$
	By \cite[Lemma~3.5]{Paksimple}, $\mu\in\Aut((F^{\circ s})^{\circ l})\subseteq\Aut_{\infty}(F)$. Then
	$$G=\mu\circ F^{\circ s}\in\left\langle\Aut_{\infty}(F),F\right\rangle\setminus\PGL_2(\C).$$
	We obtain $\left\langle\Aut_{\infty}(F),F\right\rangle=C_\infty(F)$.
	
	This establishes that $C_\infty(F)=E(F)=\left\langle\Aut_{\infty}(F),F\right\rangle\subseteq\C[z]$.
\end{proof}
For a fixed integer $m\geq2$, we consider properties which are satisfied by general polynomials of degree $m\geq2$. Here, ``general'' means that the property holds for all polynomials in a non-empty Zariski open subset of the parameter space $\Poly^m(\C)$.
\begin{Lem}\label{genpresim}
	For every integer $m\geq2$, a general polynomial $F\in\C[z]$ of degree $m$ is pre-simple.
\end{Lem}
\begin{proof}
	Up to conjugacy, we only need to consider monic centered polynomials of degree $m$, i.e., polynomials of the form
	$$F(z)=z^m+a_{m-2}z^{m-2}+\cdots+a_0,$$
	where $a_{m-2},\dots,a_0\in\C$. The parameter space of monic centered polynomials of degree $m$ is identified with
	$$\Poly_{mc}^m=\{(a_0,\dots,a_{m-2})\}\cong\A^{m-1}.$$
	The moduli space $\MPoly^m$ is the quotient of $\Poly_{mc}^m$ by the finite cyclic group $\mathbb{U}_{m-1}$ of $(m-1)$-th roots of unity (in $\C$) acting diagonally on $\A^{m-1}\cong\Poly_{mc}^m$ by
	$$\la\cdot(a_0,a_1,\dots,a_{m-2})=(\la a_0,a_1,\dots,\la^{3-m}a_{m-2}).$$
	See \cite[\S~2.1]{FG22} for more details. Let $R(t)$ be the resultant $\Res_z(F^\prime(z),F(z)-t)$ in $z$. Then
	$$R(t)=m^m\prod_{\zeta:F^\prime(\zeta)=0}(F(\zeta)-t)\in\C(a_0,\dots,a_{m-2})[t].$$
	Set $Z:=\Res_t(R(t),R^\prime(t))\in\C[a_0,\dots,a_{m-2}]$. Since there exist pre-simple polynomials in $\Poly_{mc}^m(\C)$ (for example, $z^m-mz$ when $m\geq3$, and $z^2-1$ when $m=2$), the polynomial $R(t)$ in $a_0,\dots,a_{m-2}$ is not the zero polynomial. Then $Z$ defines the empty set or a hypersurface $H$ in $\Poly_{mc}^m\cong\A^{m-1}$, exactly corresponding to (monic centered) polynomials of degree $m$ that are not pre-simple. Hence, a general polynomial of degree $m$ is pre-simple.
\end{proof}

\begin{Lem}\label{genG1}
For every integer $m\geq3$, a general polynomial $F\in\C[z]$ of degree $m$ satisfies $G(F)=1$.
\end{Lem}
\begin{proof}
Embed $\Rat_m$ into $\P^{2m+1}$ via the coefficients of rational maps of degree $m$ (see \cite{Silverman2012}). By the proof of \cite[Lemma~3.10]{Paksimple}, there exists a closed subvariety $Z$ of $\P^{2m+1}$ over $\C$ such that
$$\{F\in\Rat_m(\C):G(F)\neq1\}=Z\cap\Rat_m(\C)\quad\text{and}\quad\Poly_{mc}^m(\C)\setminus Z\neq\emptyset.$$
The conclusion follows.
\end{proof}

Lemma~\ref{genpresim}, Lemma~\ref{genG1}, and Theorem~\ref{presimrel} immediately imply: 
\begin{Thm}\label{gennosym}
	For every integer $m\geq4$, a general polynomial $F\in\C[z]$ of degree $m$ satisfies $G_0(F)=G(F)=\Aut_{\infty}(F)=E_0(F)=1$ and $C_\infty(F)=E(F)=\left\langle F\right\rangle$.
\end{Thm}
The following lemma generalizes Lemma~\ref{presimnonexcep}:
\begin{Lem}\label{presimnongenla}
Let $F\in\C[z]$ be a pre-simple polynomial of degree $m\geq4$. Then for every $r\in\Z_{>0}$, the polynomial $F^{\circ r}$ is not a generalized Latt\`es map.
\end{Lem}
\begin{proof}
Suppose $F^{\circ r}$ is a generalized Latt\`es map for some $r\in\Z_{>0}$, and let $\sO$ be an orbifold on $\widehat{\C}$ with a non-constant ramification function $\nu$ such that $F^{\circ r}:\sO\to\sO$ is a minimal holomorphic map. Set $c^*(\sO)=c(\sO)\setminus\{\infty\}$. Since $\sO$ is good, $c^*(\sO)$ is non-empty. Let $k:=\# c^*(\sO)\in\Z_{>0}$ and define
$$\mathcal{A}:=\{z\in\C\colon F^{\circ r}(z)\in c^*(\sO),\deg_z(F)=\cdots=\deg_{F^{\circ (r-1)}(z)}(F)=1\}.$$
Using the pre-simplicity of $F$, an induction on $r$ shows $\#\mathcal{A}\geq (m-2)^r k$. As $F^{\circ r}:\sO\to\sO$ is a minimal holomorphic map, we have $\mathcal{A}\subseteq c^*(\sO)$. Then
$$2k\leq(m-2)^r k\leq\#\mathcal{A}\leq\#c^*(\sO)=k,$$
which is a contradiction. Hence, $F^{\circ r}$ is not a generalized Latt\`es map for any $r\in\Z_{>0}$.
\end{proof}
\begin{Rem}\label{genlait}
	According to \cite[\S~2.3]{Pakinvvar}, the following conditions are equivalent for a rational map $f\in\C(z)$ of degree $\geq2$:
	\begin{itemize}
		\item $f$ is a generalized Latt\`es map;
		\item $f^{\circ r}$ is a generalized Latt\`es map for some $r\in\Z_{>0}$;
		\item $f^{\circ r}$ is a generalized Latt\`es map for all $r\in\Z_{>0}$.
	\end{itemize}
	Then it suffices to prove Lemma~\ref{presimnongenla} for the case $r=1$.
\end{Rem}
We now state a result on decomposing polynomials involving pre-simple ones.
\begin{Thm}\label{Thm22}
	Let $F\in\C[z]$ be a pre-simple polynomial of degree $m\geq4$. Let $G,X\in\C[z]\setminus\C$ be non-constant polynomials such that $X\circ G=F^{\circ r}\circ X$ for some $r\in\Z_{>0}$. Then $l:=\log_m(\deg(X))$ is a non-negative integer, and there exists $\mu\in\PGL_2(\C)$ such that $X=F^{\circ l}\circ\mu$ and $G=\mu^{-1}\circ F^{\circ r}\circ\mu$. If $G$ and $X$ are polynomials, we may take $\mu\in\Aff(\C)$.
\end{Thm}
\begin{proof}
If $\deg(X)=1$, then the result is trivial. Assume $\deg(X)\geq2$. By Lemma~\ref{presimnongenla}, $F^{\circ r}$ is not a generalized Latt\`es map. Applying \cite[Theorem~4.1]{Paksimple}, there exist $Y\in\C(z)\setminus\C$ and $d\in\Z_{>0}$ such that $X\circ Y=F^{\circ rd}$. By Corollary~\ref{Corpresimdecom}, $l:=\log_m(\deg(X))$ is a positive integer, and there exists $\mu\in\PGL_2(\C)$ such that $X=F^{\circ l}\circ\mu$. Then
$$F^{\circ (r+l)}\circ\mu=F^{\circ r}\circ X=X\circ G=F^{\circ l}\circ\mu\circ G.$$
By Theorem~\ref{hF}, we get $F^{\circ r}\circ\mu=\mu\circ G$, so $G=\mu^{-1}\circ F^{\circ r}\circ\mu$. If $G$ and $X$ are polynomials, then $Y$ is also a polynomial, and we may take $\mu\in\Aff(\C)$ by Corollary~\ref{Corpresimdecom}.
\end{proof}
\begin{Lem}\label{Lem23}
Let $F\in\C[z]$ be a pre-simple polynomial of degree $m\geq4$. Then for every $k\in\Z_{>0}$, the map $\gamma:\Aut(F^{\circ k})\to\Aut(F^{\circ k}),\sigma\mapsto\tau_\sigma$ (see \eqref{tausig}) is a well-defined  group automorphism.
\end{Lem}
\begin{proof}
	By Theorem~\ref{presimrel}, $\Aut_{\infty}(F)=G_0(F)$. Fix $k\in\Z_{>0}$. Then $\Aut(F^{\circ k})\subseteq\Aut_{\infty}(F)=G_0(F)$; so the map $\gamma:\Aut(F^{\circ k})\to G_0(F)$ is well-defined. For every $\nu\in\Aut(F^{\circ k})$, we have
	\begin{align*}
		F^{\circ k}\circ\gamma(\nu)\circ F&=F^{\circ k}\circ F\circ\nu=F\circ (F^{\circ k}\circ\nu)\\
		&=F\circ\nu\circ F^{\circ k}=\gamma(\nu)\circ F\circ F^{\circ k}=\gamma(\nu)\circ F^{\circ k}\circ F;
	\end{align*}
	so $F^{\circ k}\circ\gamma(\nu)=\gamma(\nu)\circ F^{\circ k}$, i.e., $\gamma(\nu)\in\Aut(F^{\circ k})$. By Lemma~\ref{gam}, the map $\gamma:\Aut(F^{\circ k})\to\Aut(F^{\circ k})$ is an injective group endomorphism on the finite group $\Aut(F^{\circ k})$, hence an automorphism.
\end{proof}
We now describe periodic curves for endomorphisms $(F_1,F_2)$ on $\P^1\times\P^1$, where $F_1,F_2\in\C[z]$ are pre-simple of the same degree $\geq4$.
\begin{Thm}\label{2presimper}
Let $F_1,F_2\in\C[z]$ be pre-simple polynomials of the same degree $d\geq4$, and let $C\subset\P^1\times\P^1$ be an irreducible algebraic curve over $\C$ that is neither a horizontal nor a vertical line. Then for every $k\in\Z_{>0}$, the following are equivalent:
\begin{enumerate}[(i)]
	\item \label{2presimper1} $(F_1,F_2)^{\circ k}(C)=C$;
	\item \label{2presimper2} There exist $s\in\Z_{\geq0}$, $\alpha\in\Aff(\C)$, and $\nu\in\Aut(F_1^{\circ k})$ such that
	$$F_2^{\circ k}=\alpha\circ F_1^{\circ k}\circ\alpha^{-1},$$
	and $C$ is one of the graphs
	$$y=(\alpha\circ\nu\circ F_1^{\circ s})(x)\quad\text{or}\quad x=(\nu\circ F_1^{\circ s}\circ\alpha^{-1})(y),$$
	where $x$ and $y$ are the coordinates on $\P^1\times\P^1$.
\end{enumerate}
\end{Thm}
\begin{proof}
	\eqref{2presimper2} $\Rightarrow$ \eqref{2presimper1} is trivial.
	
	Let $k\in\Z_{>0}$ and assume \eqref{2presimper1} holds for $k$. By Lemma~\ref{presimnongenla}, $F_1^{\circ k}$ and $F_2^{\circ k}$ are not generalized Latt\`es maps. By \cite[Theorem~1.1]{Pakinvvar}, there exist non-constant rational maps $X_1,X_2,Y_1,Y_2,B\in\C(z)\setminus\C$ and $\tilde{k}\in\Z_{\geq0}$ such that:
	\begin{itemize}
		\item $F_j^{\circ k}\circ X_j=X_j\circ B$, $B\circ Y_j=Y_j\circ F_j^{\circ k}$, $X_j\circ Y_j=F_j^{\circ k\tilde{k}}$, $Y_j\circ X_j=B^{\circ\tilde{k}}$ for $j=1,2$;
		\item The map $t\mapsto (X_1(t),X_2(t))$ parametrizes $C$.
	\end{itemize}
	
	By Theorem~\ref{Thm22}, there exist $\beta,\alpha\in\PGL_2(\C)$ and $d_1,d_2\in\Z_{\geq0}$ such that
	$$X_1=F_1^{\circ d_1}\circ\beta,\quad X_2=F^{\circ d_2}\circ\alpha,\quad\text{and}\quad\beta^{-1}\circ F_1^{\circ k}\circ\beta=B=\alpha^{-1}\circ F_2^{\circ k}\circ\alpha.$$
	Conjugating $B$ suitably, we may assume $\beta(z)=z$. Then $C$ is parametrized by $t\mapsto(F_1^{\circ d_1}(t),(F_2^{\circ d_2}\circ\alpha)(t))$ and
	$$F_1^{\circ k}=B=\alpha^{-1}\circ F_2^{\circ k}\circ\alpha.$$
	For $1\leq i\leq2$, since $F_i^{\circ k}$ is not a generalized Latt\`es map (in particular, not conjugate to $z^{d^k}$), $\infty$ is the unique totally invariant fixed point of $F_i^{\circ k}$. Hence the M\"obius transformation $\alpha$ must fix the point $\infty$, so $\alpha\in\Aff(\C)$.
	
	Since
	$$\mu_{\alpha^{-1}\circ F_2\circ\alpha}=\mu_{(\alpha^{-1}\circ F_2\circ\alpha)^{\circ k}}=\mu_{F_1^{\circ k}}=\mu_{F_1},$$
	we have $\alpha^{-1}\circ F_2\circ\alpha\in E(F_1)=\left\langle\Aut_{\infty}(F_1),F_1\right\rangle$ by Theorem~\ref{presimrel}. Also, $\Aut_{\infty}(F_1)=G_0(F)$ by Theorem~\ref{presimrel}. Hence, $\alpha^{-1}\circ F_2\circ\alpha=\tau\circ F_1$ for some $\tau\in\Aut_{\infty}(F_1)=G_0(F)$. Since
	$$F_1^{\circ k}=(\alpha^{-1}\circ F_2\circ\alpha)^{\circ k}=(\tau\circ F_1)^{\circ k},$$
	$\tau\in\Aut(F_1^{\circ k})$ by \cite[Lemma~3.5]{Paksimple}. Then $F_2=\alpha\circ\tau\circ F_1\circ\alpha^{-1}$ and $C$ is parametrized by $t\mapsto(F_1^{\circ d_1}(t),\alpha\circ(\tau\circ F_1)^{\circ d_2}(t))$. By Lemma~\ref{Lem23}, $(\tau\circ F_1)^{\circ d_2}=\tau^{\prime}\circ F_1^{\circ d_2}$ for some $\tau^{\prime}\in\Aut(F_1^{\circ k})$. Then $C$ is parametrized by
	$$t\mapsto(F_1^{\circ d_1}(t),(\alpha\circ\tau^{\prime}\circ F_1^{\circ d_2})(t)).$$
	If $d_1\leq d_2$, then $C$ is parametrized by
	$$t\mapsto(t,(\alpha\circ\tau^{\prime}\circ F_1^{\circ (d_2-d_1)})(t)),$$
	and we are done by setting $\nu:=\tau^{\prime}$ and $s:=d_2-d_1$. If $d_1>d_2$, then by Lemma~\ref{Lem23},
	$$F_1^{\circ (d_1-d_2)}\circ(\tau^\prime)^{-1}\circ\alpha^{-1}=\tau^{\prime\prime}\circ F_1^{\circ (d_2-d_1)}\circ\alpha^{-1}$$
	for some $\tau^{\prime\prime}\in\Aut(F_1^{\circ k})$. Hence $C$ is parametrized by
	$$t\mapsto((\tau^{\prime\prime}\circ F_1^{\circ (d_1-d_2)}\circ\alpha^{-1})(t),t),$$
	and we are done by setting $\nu:=\tau^{\prime\prime}$ and $s:=d_1-d_2$.
\end{proof}
Recall that two polynomials $f,g\in\C[z]$ of degree $\geq2$ are intertwined if there exists an algebraic curve $Z\subset\A^2$ (possibly reducible) whose projections to both axes are onto such that $Z$ is invariant under the endomorphism $(f,g):\A^2\to\A^2$. If $f$ and $g$ are intertwined, then $\deg(f)=\deg(g)$. The following theorem is a polynomial version of \cite[Theorem~3.3]{JX23}:
\begin{Thm}[=Theorem~\ref{poly33S1}]\label{poly33}
For every integer $d\geq4$, there exists a non-empty Zariski open subset $U$ of $\Poly^d(\C)$ such that for every $f\in U$:
\begin{enumerate}[(1)]
	\item $f$ is pre-simple;
	\item \label{poly332} for every pre-simple $g\in\Poly^d(\C)$, $f$ and $g$ are intertwined if and only if $[f]=[g]$ in $\MPoly^d(\C)$.
\end{enumerate}
\end{Thm}
\begin{proof}
By Lemma~\ref{genpresim} and Theorem~\ref{gennosym}, there exists a non-empty Zariski open subset $W_d$ of $\Poly^d(\C)$ such that for every $f\in W_d$, the polynomial $f$ is pre-simple and
\begin{equation}\label{musame}
	\{g\in\Rat_d(\C)\colon\mu_g=\mu_f\}=\{f\}.
\end{equation}

We show that $W_d$ also satisfies \eqref{poly332}. It suffices to prove the ``only if'' direction. Let $f\in W_d$ and $g\in\Poly^d(\C)$ be a pre-simple polynomial intertwined with $f$. By Theorem~\ref{2presimper}, there exist $k\in\Z_{>0}$ and $\alpha\in\Aff(\C)$ such that
$$f^{\circ k}=\alpha\circ g^{\circ k}\circ\alpha^{-1}=(\alpha\circ g\circ\alpha^{-1})^{\circ k}.$$
Then
$$\mu_f=\mu_{f^{\circ k}}=\mu_{(\alpha\circ g\circ\alpha^{-1})^{\circ k}}=\mu_{\alpha\circ g\circ\alpha^{-1}}.$$
By \eqref{musame}, $f=\alpha\circ g\circ\alpha^{-1}$. The polynomials $f$ and $g$ are conjugate as rational maps, so they are affinely conjugate, i.e., $[f]=[g]$ in $\MPoly^d(\C)$.

Hence, $W_d$ has the desired property.
\end{proof}

\subsection{The case of degree $2\leq d\leq3$}\label{secdeg23}
Having analyzed general polynomials of degree $d\geq4$ in \S~\ref{secpresimple} via the notion of pre-simplicity, we now turn to general polynomials of degree $2\leq d\leq3$, a study largely motivated by Pakovich's work \cite{Pakdeg23} on general rational maps of degree $2\leq d\leq3$. We show that for $d\in\{2,3\}$, there exists a non-empty Zariski open subset $U$ of $\Poly^d(\C)$ such that for all $(f,g)\in U\times U$, the polynomials $f$ and $g$ are intertwined if and only if they are (affinely) conjugate. This result is also a direct corollary of \cite[Theorems~3.51~and~3.52]{FG22} (see \cite[Theorem~1.4]{GNY19}).

For $d=2$, the moduli space $\MPoly^2\cong\A^1$ has dimension $1$, and every $f\in\Poly^2(\C)$ is conjugate to a unique polynomial of the form $\phi_c(z)=z^2+c$ with $c\in\C$. Thus, polynomials of degree $2$ are easier to study due to the simple normal form $\phi_c$. For $d=3$, the moduli space (and the parameter space) is also relatively simple. Note that every polynomial of degree $2\leq d\leq3$ is indecomposable because $d$ is prime.

Let $f\in\C[z]\setminus\C$ be a polynomial of degree $d\geq2$. Recall the definitions:
$$\Sigma(f):=\{\sigma\in\PGL_2(\C)\colon f\circ\sigma=f\}\quad\text{and}\quad\Sigma_\infty(f):=\bigcup_{k\geq1}\Sigma(f^{\circ k}).$$
Compared with the case $d\geq3$ (see Lemma~\ref{presimSig1}), the case $d=2$ is significantly different:
\begin{Lem}\label{lemcyc2}
For every $c\in\C^\times$, the polynomial $\phi_c(z)=z^2+c$ satisfies $\Sigma(\phi_c)=\Sigma_\infty(\phi_c)=\{z,-z\}$, which is the cyclic group of order two generated by $z\mapsto-z$.
\end{Lem}
\begin{proof}
	Fix $c\neq0$. Observe that $\phi_c$ is not conjugate to $z^2$.
	
	Clearly, $\{z,-z\}\subseteq\Sigma(\phi_c)\subseteq\Sigma_\infty(\phi_c)$. Let $\sigma\in\Sigma(\phi_c)$, and write $\sigma(z)=az+b$ with $a\in\C^\times$ and $b\in\C$ by \eqref{Siginfpoly}. The equality $\phi_c\circ\sigma=\phi_c$ implies $\sigma(z)=\pm z$. Hence, $\Sigma(\phi_c)=\{z,-z\}$.
	
	We now show by induction on $n\in\Z_{>0}$ that $\Sigma(\phi_c)=\Sigma(\phi_c^{\circ n})$ for all $n\geq1$, which implies $\Sigma(\phi_c)=\Sigma_\infty(\phi_c)$. The case $n=1$ is trivial. Assume $\Sigma(\phi_c)=\Sigma(\phi_c^{\circ (n-1)})$ for some $n\geq2$. Clearly, $\Sigma(\phi_c)\subseteq\Sigma(\phi_c^{\circ n})$. Let $\sigma\in\Sigma(\phi_c^{\circ n})$, and write $\sigma(z)=az+b$ with $a\in\C^\times$ and $b\in\C$ by \eqref{Siginfpoly}. By \cite[Lemma~3.4(ii)]{Pakdeg23} and the identity
	$$\phi_c\circ\phi_c^{\circ (n-1)}=\phi_c\circ(\phi_c^{\circ (n-1)}\circ\sigma),$$
	either $\phi_c^{\circ (n-1)}=\phi_c^{\circ (n-1)}\circ\sigma$ or $\phi_c^{\circ (n-1)}=-\phi_c^{\circ (n-1)}\circ\sigma$.
	
	If $\phi_c^{\circ (n-1)}=\phi_c^{\circ (n-1)}\circ\sigma$, then $\sigma\in\Sigma(\phi_c^{\circ (n-1)})=\Sigma(\phi_c)$, so $\Sigma(\phi_c^{\circ n})\subseteq\Sigma(\phi_c)$.
	
	Suppose $\phi_c^{\circ (n-1)}=-\phi_c^{\circ (n-1)}\circ\sigma$. Comparing the coefficients of $z^{2^{n-1}}$, $z^{2^{n-1}-1}$, and $z^{2^{n-1}-2}$ yields:
	\begin{itemize}
		\item $a^{2^{n-1}}=-1$;
		\item $b=0$ (since the term $z^{2^{n-1}-1}$ vanishes in the expansion of $\phi_c^{\circ (n-1)}$);
		\item $2^{n-2}c=-a^{2^{n-1}-2}2^{n-2}c$.
	\end{itemize}
	Since $c\neq0$, the third equality implies $a^{2^{n-1}-2}=-1$. Then
	$$-1=a^{2^{n-1}}=(a^2)^{2^{n-2}}=(a^{2^{n-1}}/a^{2^{n-1}-2})^{2^{n-2}}=(-1/-1)^{2^{n-2}}=1,$$
	which is a contradiction.
	
	Hence, $\Sigma(\phi_c^{\circ n})=\Sigma(\phi_c)$. By induction, we obtain $\Sigma(\phi_c)=\Sigma_\infty(\phi_c)$.
\end{proof}

The following two theorems show that for $2\leq d\leq3$, a general polynomial $F\in\C[z]$ of degree $d$ satisfies certain ``good'' properties.

\begin{Thm}\label{gendeg2}
	There exists a non-empty Zariski open subset $U$ of $\Poly^2(\C)$ such that for every $F\in U$:
	\begin{enumerate}[(1)]
		\item \label{gendeg2.1} $\Sigma(F)=\Sigma_\infty(F)$ is cyclic of order two;
		\item \label{gendeg2.2} if $G(z)\in\C(z)$ is a rational map of degree $2$ such that $G^{\circ k}=F^{\circ k}$ for some $k\in\Z_{>0}$, then $G=F$.
	\end{enumerate}
\end{Thm}
\begin{proof}
	Let $U:=\{a_1^2-2a_1\neq4a_0 a_2\}$, which is a non-empty Zariski open subset of
	$$\Poly^2(\C)=\{a_0z^2+a_1z+a_2\}\cong\{(a_0,a_1,a_2)\in\C^3\colon a_0\neq0\}.$$
	A simple computation shows that $a_0z^2+a_1z+a_2\in\Poly^2(\C)$ is conjugate to some $z^2+c$ with $c\in\C^\times$ if and only if its coefficients lie in $U$.
	
	Let $F(z)=a_0z^2+a_1z+a_2\in U$, and let $c=c(F)$ be the unique number in $\C^\times$ such that $F$ is conjugate to $\phi_c$. Choose $\tau\in\Aff(\C)$ with $\tau\circ F\circ\tau^{-1}=\phi_c$.
	
	\eqref{gendeg2.1} By Lemma~\ref{lemcyc2}, $\Sigma(\phi_c)=\Sigma_\infty(\phi_c)=\{z,-z\}$ is cyclic of order two, so $\Sigma(F)=\Sigma_\infty(F)$ is also cyclic of order two.
	
	\eqref{gendeg2.2} Let $G\in\C(z)$ be a rational map of degree $2$ with $G^{\circ k}=F^{\circ k}$ for some $k\in\Z_{>0}$. By the construction of $U$ and Lemma~\ref{Cinfpoly}, $G\in C(F^{\circ k})$ must be a polynomial (of degree $2$). Set $\tilde{G}=\tau\circ G\circ\tau^{-1}\in\C[z]$. Then $\tilde{G}^{\circ k}=\phi_c^{\circ k}$, and we conclude that $\tilde{G}$ (and $G$) is not conjugate to $z^2$. By \eqref{gendeg2.1} and Lemma~\ref{lemcyc2},
	$$\Sigma(\tilde{G})=\Sigma_\infty(\tilde{G})=\Sigma(\tilde{G}^{\circ k})=\Sigma(\phi_c^{\circ k})=\Sigma(\phi_c)=\{\pm z\}.$$
	By \cite[Lemma~3.5]{Pakdeg23}, $\tilde{G}=\nu\circ\phi_c$ for some $\nu\in\PGL_2(\C)$. In fact, $\nu\in\Aff(\C)$ because $\tilde{G}$ is a polynomial. From $\phi_c^{\circ k}=\tilde{G}^{\circ k}=(\nu\circ\phi_c)^{\circ k}$, we get $\phi_c^{\circ (k-1)}=(\nu\circ\phi_c)^{\circ (k-1)}\circ\nu$, so
	$$\phi_c^{\circ k}=\phi_c\circ\phi_c^{\circ (k-1)}=\phi_c\circ(\nu\circ\phi_c)^{\circ (k-1)}\circ\nu=(\phi_c\circ\nu)^{\circ k}.$$
	The same argument shows that $\Sigma(\phi_c\circ\nu)=\{\pm z\}$. On the other hand, by \eqref{Sigconj}, we have $\Sigma(\phi_c\circ\nu)=\nu^{-1}\circ\Sigma(\phi_c)\circ\nu=\{z,\nu^{-1}(-\nu(z))\}$. So $\nu^{-1}(-\nu(z))=-z$, which implies $\nu(-z)=-\nu(z)$. Thus, $\nu(z)=\la z$ for some $\la\in\C^\times$. Then
	\begin{equation}\label{compare1}
		(z^2+c)^{\circ k}=\phi_c^{\circ k}(z)=\tilde{G}^{\circ k}(z)=(\nu\circ\phi_c)^{\circ k}(z)=(\la z^2+\la c)^{\circ k}.
	\end{equation}
	Comparing the coefficient of $z^{2^k}$ on the two sides of \eqref{compare1} gives $\la^{2^k-1}=1$.
	If $k=1$, then $\la=1$.
	If $k\geq2$, the coefficient of $z^{2^k-4}$ gives
	$$2^{k-3}(2^k -2)c^2+2^{k-2}c=2^{k-3}(2^k -2)\la^{2^k-1}c^2+2^{k-2}\la^{2^k-3}c,$$
	which implies $\la^2=1$ (since $c\neq0$, $\la^{2^k-1}=1$, and $\la^{2^k-3}$ must equal 1). Then $\la=(\la^2)^{2^{k-1}}/\la^{2^k-1}=1$. In all cases, we obtain $\nu(z)=z$, so $\tilde{G}=\phi_c$ and $G=F$.
\end{proof}

\begin{Thm}\label{gendeg3}
There exists a non-empty Zariski open subset $U$ of $\Poly^3(\C)$ such that for every $F\in U$:
\begin{enumerate}[(1)]
	\item \label{gendeg3.1} $F(z)$ is pre-simple and $\Sigma(F)=\Sigma_\infty(F)=1$;
	\item \label{gendeg3.2} if $G(z)\in\C(z)$ is a rational map of degree $3$ such that $G^{\circ k}=F^{\circ k}$ for some $k\in\Z_{>0}$, then $G=F$.
\end{enumerate}
\end{Thm}
\begin{proof}
	By Lemmas~\ref{genpresim}~and~\ref{genG1}, there exists a non-empty Zariski open subset $U$ of $\Poly^3(\C)$ such that every $F\in U$ is pre-simple and $G(F)=1$. (In particular, every $F\in U$ is not conjugate to $z^3$.) Let $F\in U$. We prove \eqref{gendeg3.1} and \eqref{gendeg3.2}.
	
	\eqref{gendeg3.1} $F$ is pre-simple by assumption. Then $\Sigma(F)=\Sigma_\infty(F)=1$ by Corollary~\ref{presimSiginf1}.
	
	\eqref{gendeg3.2} Let $G\in\C(z)$ be a rational map of degree $3$ with $G^{\circ k}=F^{\circ k}$ for some $k\in\Z_{>0}$. The case $k=1$ is trivial. We assume $k\geq2$. By Lemma~\ref{Cinfpoly}, $G\in C(F^{\circ k})$ must be a polynomial. Applying Proposition~\ref{Fldecomp} to the complete decomposition $F^{\circ k}=G\circ\cdots\circ G$, we conclude that
	$$G=\nu^{-1}\circ F=\delta^{-1}\circ F\circ\nu$$
	for some $\nu,\delta\in\Aff(\C)$. (Here $\delta(z)=z$ when $k=2$.) Then $F\circ\nu=(\delta\circ\nu^{-1})\circ F$, so $\nu\in G(F)=1$. We conclude that $\nu(z)=z$ and $G=\nu^{-1}\circ F=F$.
\end{proof}

\begin{Lem}\label{gennotgenla}
	For every integer $d\geq2$, a general polynomial $F\in\C[z]$ of degree $d$ is not a generalized Latt\`es map.
\end{Lem}
\begin{proof}
For $d\geq4$, the result follows from Lemmas~\ref{presimnongenla}~and~\ref{genpresim}.

If $d=2$, then $F$ is a generalized Latt\`es map if and only if it is exceptional (see Proposition~\ref{polygenla}). Exceptional polynomials of degree $2$ are exactly those polynomials conjugate to $z^2$ or $T_2(z)=z^2-2$. Since $\MPoly^2(\C)\cong\A^1(\C)$ has dimension $1>0$, a general $F$ is not exceptional.

If $d=3$, then $F$ is exceptional if and only if it is (affinely) conjugate to $z^3$ or $\pm T_3(z)=\pm(z^3-3z)$. Thus, a general $F$ is not exceptional. By Proposition~\ref{polygenla}, a non-exceptional $F$ is a generalized Latt\`es map if and only if it is (affinely) conjugate to
$$z R_{a,b}(z)^2=z(az+b)^2$$
for some degree one polynomial $R_{a,b}(z)=az+b$ with $a,b\in\C^\times$. Thus, $F$ is a generalized Latt\`es map if and only if there exist $a,b,e\in\C^\times$ and $f\in\C$ such that
\begin{equation}\label{deg3genla}
	F(z)=(ez+f)^{-1}\circ R_{a,b}(z)\circ(ez+f)=\frac{(ez+f)(a(ez+f)+b)^2-f}{e}.
\end{equation}
Setting $\la=ae\in\C^\times$ and $\mu=af\in\C$, this becomes
\begin{equation}\label{deg3genla2}
	F(z)=\la^2z^3+\la(3\mu+2b)z^2+(\mu+b)(3\mu+b)z+\frac{\mu}{\la}(\mu^2+2\mu b+b^2-1).
\end{equation}
There are only $3$ parameters $\la,b\in\C^\times$ and $\mu\in\C$, while
$$\Poly^3(\C)=\{a_0z^3+a_1z^2+a_2z+a_3\colon(a_0,a_1,a_2,a_3)\in\C^\times\times\C^3\}$$
has dimension $4>3$. Hence, a general $F$ is not a generalized Latt\`es map.
\end{proof}

\begin{Thm}\label{deg23semiconj}
	Let $d\in\{2,3\}$. There exists a non-empty Zariski open subset $U$ of $\Poly^d(\C)$ such that for every $F\in U$, if $B,X\in\C(z)\setminus\C$ and $k\in\Z_{>0}$ satisfy $F^{\circ k}\circ X=X\circ B$, then there exist $m\in\Z_{\geq0}$ and $\mu\in\PGL_2(\C)$ such that $X=F^{\circ m}\circ\mu$ and $B=\mu^{-1}\circ F^{\circ k}\circ\mu$.
\end{Thm}
\begin{proof}
By Lemma~\ref{gennotgenla} and Theorems~\ref{gendeg2}~and~\ref{gendeg3}, there exists a non-empty Zariski open subset $U$ of $\Poly^d(\C)$ such that every $F\in U$ is not a generalized Latt\`es map and satisfies $\Sigma(F)=\Sigma_\infty(F)$.

Let $F\in U$, and let $B,X\in\C(z)\setminus\C$ be such that $F^{\circ k}\circ X=X\circ B$ for some $k\in\Z_{>0}$. If $\deg(X)=1$, then the conclusion trivially holds with $(m,\mu)=(0,X)$. We may assume $\deg(X)\geq2$. By Remark~\ref{genlait}, $F^{\circ k}$ is not a generalized Latt\`es map. Applying \cite[Proposition~3.3]{Pakinvvar}, there exist $Y\in\C(z)\setminus\C$ and $r\in\Z_{>0}$ such that $X\circ Y=F^{\circ kr}$. By Corollary~\ref{Corpresimdecom}, $m:=\log_d\deg (X)$ is a positive integer and $X=F^{\circ m}\circ\delta$ for some $\delta\in\PGL_2(\C)$. Then
\begin{equation}\label{Fkm}
	F^{\circ (k+m)}=F^{\circ k}\circ F^{\circ m}=F^{\circ k}\circ X\circ\delta^{-1}=X\circ B\circ\delta^{-1}=F^{\circ m}\circ\delta\circ B\circ\delta^{-1}.
\end{equation}
Again by Corollary~\ref{Corpresimdecom}, there exists $\nu\in\PGL_2(\C)$ such that $\delta\circ B\circ\delta^{-1}=\nu^{-1}\circ F^{\circ k}$, hence $F^{\circ (k+m)}=F^{\circ m}\circ\nu^{-1}\circ F^{\circ k}$ and $F^{\circ m}=F^{\circ m}\circ\nu$. Then $\nu\in\Sigma_\infty(F)=\Sigma(F)$. Let $\mu:=\nu\circ\delta$. Then
$$X=F^{\circ m}\circ\delta=(F^{\circ m}\circ\nu)\circ\delta=F^{\circ m}\circ\mu,$$
and
$$B=\delta^{-1}\circ\nu^{-1}\circ F^{\circ k}\circ\delta=\mu^{-1}\circ(F^{\circ k}\circ\nu)\circ\delta=\mu^{-1}\circ F^{\circ k}\circ\mu.$$
\end{proof}

\begin{Thm}\label{deg23percurve}
	Let $d\in\{2,3\}$. There exists a non-empty Zariski open subset $U$ of $\Poly^d(\C)$ such that for all $(F_1,F_2)\in U\times U$ and every irreducible algebraic curve $C\subset\P^1\times\P^1$ over $\C$ that is neither a horizontal nor a vertical line, the following are equivalent:
	\begin{enumerate}[(i)]
		\item \label{deg23percurve1} $C$ is $(F_1,F_2)$-periodic;
		\item \label{deg23percurve2} there exist $s\in\Z_{\geq0}$ and a degree one polynomial $\alpha\in\Aff(\C)$ such that
		$$F_2=\alpha\circ F_1\circ\alpha^{-1}$$
		and $C$ is one of the graphs
		$$y=(\alpha\circ F_1^{\circ s})(x)\quad\text{or}\quad x=(F_1^{\circ s}\circ\alpha^{-1})(y),$$
		where $x$ and $y$ are the coordinates on $\P^1\times\P^1$.
	\end{enumerate}
	In particular, every $(F_1,F_2)$-periodic irreducible curve is $(F_1,F_2)$-invariant.
\end{Thm}
\begin{proof}
	By Lemma~\ref{gennotgenla} and Theorems~\ref{gendeg2},~\ref{gendeg3},~and~\ref{deg23semiconj}, there exists a non-empty Zariski open subset $U$ of $\Poly^d(\C)$ such that every $F\in U$ is not a generalized Latt\`es map, satisfies properties (1) and (2) of Theorem~\ref{gendeg2} (for $d=2$) or Theorem~\ref{gendeg3} (for $d=3$), and satisfies the conclusion of Theorem~\ref{deg23semiconj}.
	
	Let $(F_1,F_2)\in U^2$, and let $C\subset\P^1\times\P^1$ be an irreducible curve over $\C$ that is neither a horizontal nor a vertical line.
	
	Assume \eqref{deg23percurve1}: $C$ is $(F_1,F_2)$-periodic. Fix $k\in\Z_{>0}$ such that $C$ is $(F_1^{\circ k},F_2^{\circ k})$-invariant. By Remark~\ref{genlait}, neither $F_1^{\circ k}$ nor $F_2^{\circ k}$ is a generalized Latt\`es map. By \cite[Theorem~1.1]{Pakinvvar}, there exist $X_1,X_2,B\in\C(z)\setminus\C$ such that:
	\begin{itemize}
		\item $F_i^{\circ k}\circ X_i=X_i\circ B$ for $i=1,2$;
		\item $t\mapsto(X_1(t),X_2(t))$ parametrizes $C$.
	\end{itemize}
	
	By Theorem~\ref{deg23semiconj}, there exist $\alpha_1,\alpha_2\in\PGL_2(\C)$ and $m_1,m_2\in\Z_{\geq0}$ such that
	$$X_1=F_1^{\circ m_1}\circ\alpha_1,\quad X_2=F_2^{\circ m_2}\circ\alpha_2,\quad\alpha_1^{-1}\circ F_1^{\circ k}\circ\alpha_1=B=\alpha_2^{-1}\circ F_2^{\circ k}\circ\alpha_2.$$
	Let $\alpha=\alpha_2\circ\alpha_1^{-1}$. Then
	$$F_2^{\circ k}=\alpha\circ F_1^{\circ k}\circ\alpha^{-1}=(\alpha\circ F_1\circ\alpha^{-1})^{\circ k},$$
	so $F_2=\alpha\circ F_1\circ\alpha^{-1}$ by (2) of Theorem~\ref{gendeg2} or~\ref{gendeg3}. For $1\leq i\leq2$, since $F_i$ is not a generalized Latt\`es map (in particular, not conjugate to $z^d$), $\infty$ is the unique totally invariant fixed point of $F_i$. Hence, the M\"obius transformation $\alpha$ must fix $\infty$, so $\alpha\in\Aff(\C)$.
	
	The curve $C$ is also parametrized by $t\mapsto(F_1^{\circ m_1}(t),\alpha\circ F_1^{\circ m_2}(t))$. If $m_1\geq m_2$, then $t\mapsto(F_1^{\circ s}\circ\alpha^{-1}(t),t)$ parametrizes $C$, where $s:=m_1-m_2\in\Z_{\geq0}$. If $m_1< m_2$, then $t\mapsto(t,\alpha\circ F_1^{\circ s}(t))$ parametrizes $C$, where $s:=m_2-m_1\in\Z_{>0}$. Thus, $C$ is one of the graphs in \eqref{deg23percurve2}.
	
	Conversely, if \eqref{deg23percurve2} holds, then $C$ is clearly $(F_1,F_2)$-invariant.
\end{proof}
\begin{Rem}
	Let $d\in\{2,3\}$. Compared to Theorem~\ref{2presimper}, the form of periodic curves in Theorem~\ref{deg23percurve} is relatively simple. This is mainly due to (2) of Theorem~\ref{gendeg2} or~\ref{gendeg3}, which implies:
	\begin{itemize}
		\item for two general polynomials $F_1,F_2\in\Poly^d(\C)$, $F_1$ and $F_2$ are conjugate if and only if $F_1^{\circ k}$ and $F_2^{\circ k}$ are conjugate for some $k\in\Z_{>0}$;
		\item for a general polynomial $F\in\Poly^d(\C)$, $\Aut(F)=\Aut_\infty(F)$.
	\end{itemize}
\end{Rem}

\begin{Prop}[=Proposition~\ref{poly33deg23intro}]\label{poly33deg23}
	Let $d\in\{2,3\}$. There exists a non-empty Zariski open subset $U$ of $\Poly^d(\C)$ such that for all $(f,g)\in U\times U$, $f$ and $g$ are intertwined if and only if $[f]=[g]$ in $\MPoly^d(\C)$.
\end{Prop}
\begin{proof}
Take $U$ as in Theorem~\ref{deg23percurve}. The result follows directly.
\end{proof}

\subsection{Proof of Theorem~\ref{polygeninj}}\label{S2.3}
Using Theorem~\ref{poly33} or Proposition~\ref{poly33deg23} in place of \cite[Theorems~3.51~and~3.52]{FG22}, we obtain a proof of Theorem~\ref{polygeninj} following Ji and Xie \cite{JX23}.
\begin{proof}[Proof of Theorem~\ref{polygeninj}]
	Let $d\geq2$ be an integer. Suppose, for the sake of contradiction, that $\tilde{\tau}_d$ is not generically injective. By Theorem~\ref{poly33} or Proposition~\ref{poly33deg23}, there exists a non-empty Zariski open subset $U$ of $\MPoly^d(\C)$ such that for all $f,g\in\Poly^d(\C)$ with $[f],[g]\in U$, 
	\begin{equation}\label{Uprop}
		f\text{ and }g\text{ are intertwined if and only if }[f]=[g].
	\end{equation}
	Let $W$ be a non-empty Zariski open subset of the Zariski closure of $\tilde{\tau}_d(U)$ such that $\tilde{\tau}_d^{-1}(W)\subseteq U$ and $\tilde{\tau}_d:\tilde{\tau}_d^{-1}(W)\to W$ is a finite \'etale morphism of degree at least $2$. After shrinking $U$, we may assume $U=\tilde{\tau}_d^{-1}(W)$.
	
	It is well-known that the PCF locus $\{f\in\Poly^d(\C)\colon f\text{ is PCF}\}$ is Zariski-dense in $\Poly^d(\C)$. As in the proof of \cite[Theorem~1.3]{JX23}, we can find two non-isotrivial algebraic families $h_1,h_2$ of degree $d$ polynomials parametrized by an irreducible affine curve $C$ with $h_i(C(\C))\subseteq U$ for $1\leq i\leq2$ and
	$$\tilde{\tau}_d\circ h_1=\tilde{\tau}_d\circ h_2\colon C\to\A_\C^N$$
	(where $N=N_{d,1}+\cdots+N_{d,m_d}\in\Z_{>0}$) such that
	\begin{equation}\label{noconj}
		\text{for every }t\in C(\C)\text{, }h_1(t)\neq h_2(t)\text{ in }\MPoly^d\text{,}
	\end{equation}
	and the PCF locus
	$$\{t\in C(\C)\colon h_1(t)\text{ is PCF}\}$$
	of $h_1$ is Zariski-dense in $C$.
	
	By \cite[Theorem~1.14]{JXZ}, the property of being PCF is determined by the length spectrum (in particular, by the multiplier spectrum). So
	$$\{t\in C(\C)\colon h_2(t)\text{ is PCF}\}=\{t\in C(\C)\colon h_1(t)\text{ is PCF}\},$$
	which is Zariski-dense in $C$. By the DAO-type theorem \cite[Theorem~1.5]{JX23} proved by Ji and Xie, we deduce that
	$$\#\{t\in C(\C)\colon h_1(t)\text{ and }h_2(t)\text{ are intertwined}\}=\infty.$$
	Note that polynomials of degree $\geq2$ cannot be Latt\`es, since the Julia set of any Latt\`es map is the entire Riemann sphere \cite[\S~2]{milnor2006lattes}. By \eqref{Uprop}, we obtain
	$$\#\{t\in C(\C)\colon h_1(t)=h_2(t)\text{ in }\MPoly^d\}=\infty,$$
	contradicting \eqref{noconj}.
\end{proof}

\subsection{General polynomial pairs and the {Z}ariski-dense orbit conjecture}\label{secZDO}
As a byproduct, we present a proof of the Zariski-dense orbit conjecture for a general split polynomial endomorphism on $\A^2$ with all factors of the same degree $d\geq2$, which is essentially due to Pakovich \cite{Paksimple,Pakdeg23}.
\begin{Conj}[Zariski-dense orbit conjecture~=~ZDO]\label{ZDO}
	Let $k$ be an algebraically closed field of characteristic zero. Let $X$ be an irreducible quasi-projective variety over $k$ and $f$ a dominant rational self-map on $X$. If $\{g\in k(X)\colon g\circ f=g\}=k$, where $k(X)$ is the function field of $X$, then there exists $x\in X(k)$ whose forward orbit under $f$ is well-defined and Zariski-dense in $X$.
\end{Conj}
We remark that the converse of ZDO holds and is easy to prove.
\begin{Thm}\label{presimapp}
	Fix an integer $d\geq2$. For a general pair of polynomials $(F_1,F_2)\in\Poly^d(\C)^2$, the endomorphism $(F_1,F_2):\A^2\to\A^2$ has no irreducible periodic curves other than vertical or horizontal lines. In particular, for all $x\in\A^1(\C)\setminus\PrePer(F_1)$ and $y\in\A^1(\C)\setminus\PrePer(F_2)$, the forward orbit $O_{(F_1,F_2)}((x,y))$ is Zariski-dense in $\A^2$.
\end{Thm}
\begin{proof}
Observe that for a general pair of polynomials $(F_1,F_2)\in\Poly^d(\C)^2$, we have $[F_1]\neq[F_2]$ in $\MPoly^d$. For $d\in\{2,3\}$, the result follows from Theorem~\ref{deg23percurve}.

Assume $d\geq4$. Let $F_1,F_2$ be pre-simple polynomials of degree $d$ such that
$$G(F_1)=G(F_2)=1\quad\text{and}\quad[F_1]\neq[F_2].$$
By Lemmas~\ref{genpresim}~and~\ref{genG1}, such pairs are general in $\Poly^d(\C)\times\Poly^d(\C)$.

Suppose that $C\subset\A^2$ is an irreducible $(F_1,F_2)$-periodic curve over $\C$ that is neither a horizontal nor a vertical line. By Theorem~\ref{2presimper}, $F_2^{\circ k}=\alpha\circ F_1^{\circ k}\circ\alpha^{-1}$ for some $k\in\Z_{>0}$ and $\alpha\in\Aff(\C)$. By Theorem~\ref{presimrel} and the fact that $G_0(F_2)\subseteq G(F_2)=1$, we have $E(F_2)=\left\langle F_2\right\rangle$. Since
$$\mu_{\alpha\circ F_1\circ\alpha^{-1}}=\mu_{(\alpha\circ F_1\circ\alpha^{-1})^{\circ k}}=\mu_{F_2^{\circ k}}=\mu_{F_2},$$
we see that $\alpha\circ F_1\circ\alpha^{-1}\in E(F_2)=\left\langle F_2\right\rangle$. Comparing degrees gives $\alpha\circ F_1\circ\alpha^{-1}=F_2$, contradicting $[F_1]\neq[F_2]$.
\end{proof}
\begin{Rem}
	Let $F_i\in\C[z]$ be a polynomial of degree $d_i\geq2$ for $1\leq i\leq2$ such that $d_1\neq d_2$. In the non-equal-degree case, $F_1$ and $F_2$ are not intertwined (see Remark~\ref{interwdef}), so the endomorphism $(F_1,F_2):\A^2\to\A^2$ has no irreducible periodic curves other than vertical or horizontal lines.
\end{Rem}
\begin{Rem}
	Note that for any polynomial $F(z)\in k[z]$ of degree $d\geq2$ over a field $k$ of characteristic zero, we have $\A^1(k)\setminus\PrePer(F_1)\neq\emptyset$. It is easy to see that for a general pair $(F_1,F_2)$ of polynomials of degree $d\geq2$ defined over $\overline{\Q}$, the analogous conclusion holds over $\overline{\Q}$. 
\end{Rem}
\begin{Rem}
	Xie \cite{Xie2017} proved ZDO for all dominant polynomial endomorphisms $f:\A_k^2\to\A_k^2$ over an algebraically closed field $k$ of characteristic zero using valuative techniques.
\end{Rem}

\section{Isospectral polynomials}\label{S3}
\subsection{Marked critical points and a theorem of Favre-Gauthier}\label{S3.1}
On the parameter space $\Poly^d$ (or the moduli space $\MPoly^d$), the critical points are not given by global morphisms. This causes difficulties when working with information about critical points (e.g., PCF maps and persistently preperiodic critical points). To address this, it is necessary to consider critically marked polynomials via a base change of the parameter space (or the moduli space). In this subsection, we review classical constructions of critically marked polynomials; see \cite[\S~2.1]{FG22} and \cite[\S~5]{DF08}.

A \emph{critically marked polynomial} of degree $d\geq2$ is a tuple $(P,c_0,\dots,c_{d-2})$ where $P\in\Poly^d$ and $c_0,\dots,c_{d-2}$ are the critical points (other than $\infty$) of $P$, counted with multiplicity. The space of critically marked polynomials of degree $d\geq2$ is the closed subvariety $\Pcrit^d$ of $\Poly^d\times\A^{d-1}$ given by
$$\left\lbrace(P=\sum_{i=0}^d a_iz^i,c_0,c_1,\dots,c_{d-2})\in\Poly^d\times\A^{d-1}\colon\sum_{j=1}^d ja_j z^{j-1}=d a_d\prod_{k=0}^{d-2}(z-c_k)\right\rbrace.$$
Let $\MPcrit^d$ denote the quotient $\Pcrit^d/\Aff$, where the action is given by
$$\sigma\cdot\left(P,c_0,\dots,c_{d-2}\right)=\left(\sigma\circ P\circ\sigma^{-1},\sigma(c_0),\dots,\sigma(c_{d-2})\right),$$
for $\sigma\in\Aff$ and $\left(P,c_0,\dots,c_{d-2}\right)\in\Pcrit^d$. There is a canonical forgetful morphism $\MPcrit^d\to\MPoly^d$. 

Define the morphism
$$\eta:\A^{d-1}\to\MPcrit^d,(c,a)=(c_1,\dots,c_{d-2},a)\mapsto\left[\left(P_{c,a}(z),0,c_1,\dots,c_{d-2}\right)\right],$$
where
$$P_{c,a}(z)=\frac{1}{d}z^d+\sum_{j=2}^{d-1}(-1)^{d-j}\sigma_{d-j}(c)\frac{z^j}{j}+a^d$$
and $\sigma_j(c)$ is the $j$-th elementary symmetric polynomial in $(c_1,\dots,c_{d-2})$ for $1\leq j\leq d-2$. Note that the (finite) critical points of $P_{c,a}$ are exactly $c_0:=0,c_1,\dots,c_{d-2}$. The morphism $\eta:\A^{d-1}\to\MPcrit^d$ is defined over $\Q$ and is a finite ramified covering of degree $d(d-1)$ (see \cite[Proposition~5.1]{DF08}). Let $\eta_0:\A^{d-1}\to\Poly^d$ be the morphism defined by $\eta_0(c,a)=P_{c,a}$.

The following theorem is a part of \cite[Theorem~B]{FG22} and is essential for the proof of Theorem~\ref{polynoninj}.
\begin{Thm}[Favre and Gauthier]\label{FG}
	Let $(P,a)$ and $(Q,b)$ be non-exceptional active dynamical pairs parametrized by an irreducible algebraic curve $C$, of respective degrees $d,\delta\geq2$, defined over a number field $K$. Then the following conditions are equivalent:
	\begin{enumerate}[(1)]
		\item \label{FG2} The set $\{t\in C(\overline{\Q})\colon a(t)\in\PrePer(P(t))\text{ and }b(t)\in\PrePer(Q(t))\}$ is infinite;
		\item There exist integers $N,M\geq1$ and families $R,\pi,\tau$ of polynomials of degree $\geq1$ parametrized by $C$ such that
		$$\tau\circ P^{\circ N}=R\circ\tau\quad\text{and}\quad\pi\circ Q^{\circ M}=R\circ\pi.$$
	\end{enumerate}
\end{Thm}
Here, a \emph{dynamical pair} parametrized by $C$ of degree $d$ is a pair $(P,a)$ where $P:C\to\Poly^d$ is a family of polynomials of degree $d$ over $C$, and $a:C\to\A^1$ is a morphism (i.e., a marked point of the family $P$). The dynamical pair $(P,a)$ is called \emph{active} (or \emph{not persistently preperiodic}) if for all integers $n>m\geq0$, we have $(P(t))^{\circ n}(a(t))\neq (P(t))^{\circ m}(a(t))$ in the function field of $C$. It is called \emph{non-exceptional} if the polynomial $P(t)$ is non-exceptional for every $t\in U(\overline{\Q})$, where $U$ is a suitable non-empty Zariski open subset of $C$. If two pairs $(P,a)$ and $(Q,b)$ satisfy the property in \eqref{FG2}, then they are called \emph{entangled}.

\subsection{Description of the non-injective locus}\label{S3.2}
We prove Theorem~\ref{polynoninj}, relying on results of Pakovich and Theorem~\ref{FG} of Favre-Gauthier.

Following the proof of \cite[Lemma~3.5]{JX23}, we obtain the following elementary lemma:
\begin{Lem}\label{sac}
Let $f$ be a polynomial map of degree $d\geq2$. Then its multiplier spectrum $(S_n(f))_{n=1}^\infty$ determines the number $\sac(f)\in\{0,1,\dots,d-1\}$ of superattracting cycles of $f$ in $\C$ (counted without multiplicity).
\end{Lem}
Since $\infty$ is a totally ramified fixed point of $f$, the bound $\sac(f)\leq d-1$ follows from the Riemann–Hurwitz formula. Note that $\sac(\cdot)$ descends to the moduli space $\MPoly^d$.

\begin{Prop}\label{multidet}
	Let $\phi:C\to\Poly^d$ be a non-isotrivial family of polynomials of degree $d\geq2$, parametrized by an irreducible curve $C$. Assume that $\phi$ factors through $\eta_0:\A^{d-1}\to\Poly^d,(c,a)\mapsto P_{c,a}$ via some morphism $\psi:C\to\A^{d-1}$. Fix such a $\psi$. Let $m(\phi)\geq0$ be the smallest nonnegative integer $m$ such that
	$$\{t\in C(\C)\colon\sac(\phi(t))=m\}$$
	is infinite. Define
	$$Z(\phi):=\{t\in C(\C)\colon\sac(\phi(t))>m(\phi)\},$$
	and let $Z^\prime(\phi)$ be the set of all $t\in C(\C)$ such that $c_j\circ\psi$ is not a persistently $\phi$-preperiodic marked point but $c_j(\psi(t))\in\Per(\phi(t))$ for some $0\leq j\leq d-2$. Here $c_j$ is the coordinate map $\A^{d-1}\to\A^1,(c,a)\mapsto c_j$ for $1\leq j\leq d-2$, and $c_0=0:\A^{d-1}\to\A^1$.
	
	Then $m(\phi)$ and $Z(\phi)$ depend only on the multiplier spectrum $(S_n(\phi(t)))_{n\geq1,t\in C(\C)}$. Moreover, $Z(\phi)$ is an infinite subset of $C(\C)$ and the symmetric difference
	$$Z(\phi)\Delta Z^\prime(\phi)=(Z(\phi)\cup Z^\prime(\phi))\setminus(Z(\phi)\cap Z^\prime(\phi))$$
	is finite.
\end{Prop}
\begin{proof}
	By Lemma~\ref{sac}, the number $m(\phi)$ depends only on the multiplier spectrum, and hence so does $Z(\phi)$. Note that $m(\phi)$ does not change after removing finitely many points of $C(\C)$. Hence it suffices to prove that $Z(\phi)=Z^\prime(\phi)$ is infinite, after removing finitely many points of $C(\C)$.
	
	For simplicity, we denote the marked point $c_j\circ\psi:C\to\A^1$ of $\phi$ by $a_j$, for $0\leq j\leq d-2$.
	
	Let $(i,j)\in\{0,\dots,d-2\}^2$ be a pair of indices such that both $a_i$ and $a_j$ are persistently $\phi$-periodic. Let $n,m\in\Z_{>0}$ be minimal such that
	$$\phi(t)^{\circ n}(a_i(t))=a_i(t)\quad\text{and}\quad\phi(t)^{\circ m}(a_j(t))=a_j(t)$$
	in the function field of $C$. For every integer $0\leq r< n$, the subset
	$$\{t\in C(\C)\colon a_j(t)=\phi(t)^{\circ r}(a_i(t))\}$$
	is either finite or all of $C(\C)$. After removing finitely many points of $C(\C)$, we may assume that for all such pairs $(i,j)$, either
	\begin{equation}\label{ijorbsame}
		O_{\phi(t)}(a_i(t))=O_{\phi(t)}(a_j(t))\quad\text{for all }t\in C(\C),
	\end{equation}
	or
	\begin{equation}\label{ijorbnotinter}
		O_{\phi(t)}(a_i(t))\cap O_{\phi(t)}(a_j(t))=\emptyset\quad\text{for all }t\in C(\C).
	\end{equation}
	Here for $g\in\C(z)$ and $w\in\P^1(\C)$, $O_g(w)$ denotes the forward $g$-orbit of $w$. In the first case \eqref{ijorbsame}, we write $i\sim j$. Let
	$$P(\phi):=\{0\leq i\leq d-2\colon a_i\text{ is persistently }\phi\text{-periodic}\}.$$
	Then $\sim$ is an equivalence relation on $P(\phi)$.
	
	Let $i\in\{0,\dots,d-2\}$ be an index such that $a_i$ is persistently $\phi$-preperiodic but not persistently $\phi$-periodic. Let $k>l\geq0$ be minimal integers such that
	$$\phi(t)^{\circ k}(a_i(t))=\phi(t)^{\circ l}(a_i(t))$$
	in the function field of $C$. Then $l>0$. For every integer $l\leq r< k$, the subset
	$$\{t\in C(\C)\colon a_i(t)=\phi(t)^{\circ r}(a_i(t))\}$$
	is finite, since $a_i$ is not persistently periodic. After removing finitely many points of $C(\C)$, for all such indices $i$, we may assume that
	\begin{equation}\label{stricprepernoper}
		a_i(t)\notin\Per(\phi(t))\quad\text{for all }t\in C(\C).
	\end{equation}
	
	It is well-known that PCF parameters do not form families in $\MPoly^d$. Since $\phi$ is non-isotrivial, we conclude that
	\begin{equation}\label{Anonempty}
		A(\phi):=\{0\leq i\leq d-2\colon a_i\text{ is not persistently }\phi\text{-preperiodic}\}\neq\emptyset.
	\end{equation}
	After removing finitely many points of $C(\C)$ and arguing similarly, we may assume that
	\begin{equation}\label{APorbnotinter}
		c_i(t)\notin O_{\phi(t)}(c_k(t))
	\end{equation}
	for all $(i,k)\in A(\phi)\times P(\phi)$ and all $t\in C(\C)$. Fix any $i\in A(\phi)$. By \cite[Lemma~2.4]{JX23},
	\begin{equation}\label{Ainfper}
		\#\{t\in C(\C)\colon a_i(t)\in\Per(\phi(t))\}=+\infty.
	\end{equation}
	On the other hand, since $a_i$ is not persistently $\phi$-preperiodic, we have
	$$\#\{t\in C(\C)\colon\phi(t)^{\circ n}(a_i(t))=a_i(t)\}<+\infty,$$
	for every $n\in\Z_{>0}$. So $\{t\in C(\C)\colon a_i(t)\in\Per(\phi(t))\}$ is countable. Due to the uncountability of $C(\C)$, 
	\begin{equation}\label{nonprepernonperinf}
		\#\bigcap_{k\in A(\phi)}\{t\in C(\C)\colon a_k(t)\notin\Per(\phi(t))\}=+\infty.
	\end{equation}
	
	Based on the case analysis above, it is evident that
	$$m(\phi)=\#(P(\phi)/\sim).$$
	Indeed, we have $\sac(\phi(t))\geq\#(P(\phi)/\sim)$ for all $t\in C(\C)$ by the analysis of persistently $\phi$-periodic $a_i$’s, hence $m(\phi)\geq\#(P(\phi)/\sim)$. By \eqref{stricprepernoper} and \eqref{nonprepernonperinf}, there are infinitely many (in fact, uncountably many) $t\in C(\C)$ such that $\sac(\phi(t))=\#(P(\phi)/\sim)$. Hence $m(\phi)=\#(P(\phi)/\sim)$ by definition.
	
	By \eqref{Anonempty} and \eqref{Ainfper}, $Z^\prime(\phi)$ is infinite. For every $t\in Z^\prime(\phi)$, from \eqref{APorbnotinter} we see that
	$$\sac(\phi(t))\geq\#(P(\phi)/\sim)+1>m(\phi),$$
	so $Z^\prime(\phi)\subseteq Z(\phi)$; hence $Z(\phi)$ is also infinite. Conversely, for every $t\in C(\C)\setminus Z^\prime(\phi)$, by \eqref{stricprepernoper} we have $\sac(\phi(t))=\#(P(\phi)/\sim)=m(\phi)$, so $Z(\phi)\subseteq Z^\prime(\phi)$. Therefore $Z(\phi)=Z^\prime(\phi)$, completing the proof. 
\end{proof}

\begin{proof}[Proof of Theorem~\ref{polynoninj}]
After a suitable base change of $C$ (and restricting to a suitable non-empty affine open subset) over $\overline{\Q}$, we may assume that the following hold:
\begin{itemize}
	\item for every $t\in C(\C)$, both $\phi_1(t)$ and $\phi_2(t)$ are non-exceptional, and
	$$([\phi_1(t)],[\phi_2(t)])\in\PNI_d;$$
	\item for $1\leq j\leq2$, the morphism $\Psi\circ\phi_j:C\to\MPoly^d$ factors through $\Psi\circ\eta_0:\A^{d-1}\to\MPoly^d$ via some morphism $\psi_j:C\to\A^{d-1}$ over $\overline{\Q}$.
\end{itemize}
If $K$ is a number field such that $\phi_1$, $\phi_2$, and $C$ are defined over $K$, then one can verify that the base change of $C$, $\psi_1$, and $\psi_2$ in above reduction can also be defined over $K$. As all the properties we consider (multiplier spectrum, equivalence, exceptionalness etc.) do not change under conjugacy, we may replace $\phi_j$ by $\eta_0\circ\psi_j$ and assume that $\phi_j:C\to\Poly^d$ factors through $\eta_0:\A^{d-1}\to\Poly^d$ by $\psi_j$, for $1\leq j\leq2$.

By Proposition~\ref{multidet}, we obtain:
\begin{itemize}
	\item $m(\phi_1)=m(\phi_2)$ and $Z(\phi_1)=Z(\phi_2)$;
	\item $\# Z^\prime(\phi_1)=\# Z^\prime(\phi_2)=+\infty$ and $\#\left(Z^\prime(\phi_1)\Delta Z^\prime(\phi_2)\right)<+\infty$.
\end{itemize}
For $1\leq j\leq2$ and
$$k\in A(\phi_j)=\{0\leq l\leq d-2\colon c_l\circ\psi_j\text{ is not persistently }\phi_j\text{-preperiodic}\},$$
set
$$Z_{jk}:=\{t\in C(\C)\colon c_k(\psi_j(t))\in\Per(\phi_j(t))\}=\{t\in C(\overline{\Q})\colon c_k(\psi_j(t))\in\Per(\phi_j(t))\},$$
where the last equality is because all the objects are defined over $\overline{\Q}$. By definition, $Z^\prime(\phi_j)=\bigcup_{k\in A(\phi_j)}Z_{jk}$ for $1\leq j\leq2$. Since the intersection of $Z^\prime(\phi_1)$ and $Z^\prime(\phi_2)$ is infinite, we can choose $(k,l)\in A(\phi_1)\times A(\phi_2)$ such that $\#(Z_{1k}\cap Z_{2l})=+\infty$. After renumbering, assume $k=l=0$. Then for $1\leq j\leq2$, the marked critical point $0$ is not persistently $\phi_j$-preperiodic, hence $(\phi_j,0)$ is an active dynamical pair. We have
$$\#\{t\in C(\overline{\Q})\colon 0\in\Per(\phi_1(t))\cap\Per(\phi_2(t))\}=+\infty,$$
hence the pairs $(\phi_1,0)$ and $(\phi_2,0)$ are entangled in the sense of Favre and Gauthier \cite{FG22}.

Applying Theorem~\ref{FG}, there exist $N,M\in\Z_{>0}$ and families $R,\tau,\pi$ of polynomials of degree $\geq1$ parametrized by $C$ such that
$$\tau\circ\phi_1^{\circ N}=R\circ\tau\quad\text{and}\quad\pi\circ\phi_2^{\circ M}=R\circ\pi.$$
Note that $d^N=\deg(\phi_1^{\circ N})=\deg(R)=\deg(\phi_2^{\circ M})=d^M$, so $N=M$.

Fix $t\in C(\C)$. Write $h_1=\phi_1(t)^{\circ N}$ and $h_2=\phi_2(t)^{\circ N}$. For simplicity, we abuse the notation and write $R,\tau,\pi$ for $R(t),\tau(t),\pi(t)$, respectively. Then
$$\tau\circ h_1=R\circ\tau\quad\text{and}\quad\pi\circ h_2=R\circ\pi.$$
Assume that $h_1$ and $h_2$ are not equivalent, i.e., the statement \eqref{polynoninj1} of Theorem~\ref{polynoninj} does not hold for $N$ and the fixed $t\in C(\C)$.

\medskip

\textbf{Claim:} There exist integers $k_1,k_2,l>0$, a non-constant polynomial $V(z)\in\C[z]\setminus\C$, and $\zeta\in\C$, with
$$2\leq k:=\lcm(k_1,k_2)\leq d^N,\quad\gcd(d,k)=1,\quad V(0)\neq0,\quad\zeta^{k_2}=1,$$
such that
$$h_1\sim z^l\cdot V(z^{k_1})^{k_1^\prime}\quad\text{and}\quad h_2\sim\zeta z^l\cdot V(z^{k_2})^{k_2^\prime},$$
where $k_j^\prime=k/k_j$ for $1\leq j\leq2$.
	
We prove the Claim using results of Pakovich and elementary argument. Since $h_1\not\sim h_2$, we have $h_j\not\sim R$ for some $1\leq j\leq2$. Without loss of generality, assume $h_1\not\sim R$. By Theorem~\ref{genla32}, $R$ is a generalized Latt\`es map and there exist $\tau_0,\tau_1,\tilde{h}_1\in\C(z)\setminus\C$ satisfying:
\begin{itemize}
	\item $\tau=\tau_0\circ\tau_1$ and $k_1:=\deg(\tau_0)\geq2$;
	\item $\tilde{h}_1\circ\tau_1=\tau_1\circ h_1$, $R\circ\tau_0=\tau_0\circ\tilde{h}_1$, and $h_1\sim\tilde{h}_1$;
	\item $(f=R,p=\tau_0,g=\tau_0,q=\tilde{h}_1)$ is a good solution of $f\circ p=g\circ q$;
	\item Both $R:\sO_2^{\tau_0}\to\sO_2^{\tau_0}$ and $\tilde{h}_1:\sO_1^{\tau_0}\to\sO_1^{\tau_0}$ are minimal holomorphic maps between orbifolds;
	\item There exists $s\in\Z_{>0}$ such that $\tau_1$ is a compositional right factor of $h_1^{\circ s}$ and a compositional left factor of $\tilde{h}_1^{\circ s}$.
\end{itemize}
After conjugacy by M\"obius transformations if necessary, we may assume that $\tau_1$, $\tau_0$, and $\tilde{h}_1$ are polynomials.

We show that, up to conjugacy by a degree one polynomial, we may assume $\tau_0(z)=z^{k_1}$. If $R$ is exceptional, then by \cite[Theorem~3.39]{FG22}, $h_1=\phi_1(t)^{\circ N}$ and $\phi_1(t)$ are also exceptional, contradicting our assumption. Thus $R$ is a non-exceptional polynomial that is also a generalized Latt\`es map. Since $\tau_0:\sO_1^{\tau_0}\to\sO_2^{\tau_0}$ is a covering map between orbifolds, the Riemann–Hurwitz formula gives
$$\chi(\sO_2^{\tau_0})=\frac{\chi(\sO_1^{\tau_0})}{k_1}\leq\frac{2}{k_1}\leq1<2.$$
In particular, the ramification function of $\sO_2^{\tau_0}$ is not identically $1$. By Lemma~\ref{polygenla}, we have $\nu(\sO_2^{\tau_0})=\{n,n\}$ for some integer $n\geq2$. By the definition of $\sO_2^{\tau_0}$, the point $\infty$ shows $k_1=\deg(\tau_0)\in\nu(\sO_2^{\tau_0})$, hence $n=k_1$. Since $\nu(\sO_2^{\tau_0})=\{k_1,k_1\}$, the polynomial $\tau_0$ has exactly one critical value in $\C$, hence is of the form $\tau_0(z)=az^{k_1}+c$, where $a\in\C^\times$ and $c\in\C$. Replacing the triple $(\tau_0,R,\tau)$ with $(L\circ\tau_0,L\circ R\circ L^{-1},L\circ\tau)$, where $L(z)=(z-c)/a$, we may assume $\tau_0(z)=z^{k_1}$.

By Theorem~\ref{zngood} and Remark~\ref{zngoodpoly}, there exists $R_1(z)\in\C[z]\setminus\{0\}$ such that
$$R(z)=z^{r_1} R_1(z)^{k_1}\quad\text{and}\quad\tilde{h}_1(z)=z^{r_1} R_1(z^{k_1}),$$
for some integer $r_1\geq1$ with $\gcd(r_1,k_1)=1$. Since $R$ is non-exceptional, we have $\deg(R_1)\geq1$.

Similarly, if $h_2\not\sim R$, then there exist $R_2(z),\tilde{h}_2\in\C[z]\setminus\C$ such that
$$h_2\sim\tilde{h}_2,\quad R(z)=z^{r_2} R_2(z)^{k_2},\quad\tilde{h}_2(z)=z^{r_2} R_2(z^{k_2}),$$
for some integers $r_2\geq1$ and $k_2\geq2$ with $\gcd(r_2,k_2)=1$. If $h_2\sim R$, then we set $\tilde{h}_2=R_2=R$, $r_2=0$, and $k_2=1$.

Set $k=\lcm(k_1,k_2)\geq k_1\geq2$ and $k_j^\prime=k/k_j$ for $1\leq j\leq2$. Then
\begin{equation}\label{tildehj}
	\tilde{h}_1(z^{k_1^\prime})^{k_1}=R(z^k)=\tilde{h}_2(z^{k_2^\prime})^{k_2}.
\end{equation}
Write $R_j(z)=z^{l_j}V_j(z)$ with $l_j\in\Z_{\geq0}$ and $V_j(z)\in\C[z]$ satisfying $V_j(0)\neq0$, for $1\leq j\leq2$. Then \eqref{tildehj} becomes
$$z^{(r_1+k_1 l_1)k}\cdot V_1(z^k)^{k_1}=z^{(r_2+k_2 l_2)k}\cdot V_2(z^k)^{k_2}.$$
Comparing the order at $0$, we get $r_1+k_1 l_1=r_2+k_2 l_2=:l\in\Z_{>0}$. We have $V_1(z^k)^{k_1}=V_2(z^k)^{k_2}$, so $V_1(z)^{k_1}=V_2(z)^{k_2}$. Examining irreducible factors of $V_1(z)$, there exists $V(z)\in\C[z]$ such that $V_1(z)=V(z)^{k_1^\prime}$. Then $V_2(z)^{k_2}=V(z)^k=V(z)^{k_2^\prime k_2}$, which implies that $V_2(z)=\zeta\cdot V(z)^{k_2^\prime}$ for some $k_2$-th root of unity $\zeta\in\C^\times$. Therefore
$$\tilde{h}_1(z)=z^l\cdot V(z^{k_1})^{k_1^\prime}\quad\text{and}\quad\tilde{h}_2(z)=\zeta\, z^l\cdot V(z^{k_2})^{k_2^\prime}.$$
Note that $\gcd(r_1,k_1)=1$ implies $\gcd(d,k_1)=1$, since
$$d^N=\deg(R)=r_1+k_1\cdot\deg(R_1).$$
Similarly, $\gcd(d,k_2)=1$. Then $\gcd(d,k)=1$. Since $R(z)=z^l V(z)^k$ is non-exceptional, we have $\deg(V)\geq1$. In particular,
$$d^N=\deg(R)=l+k\deg(V)\geq r_1+k\cdot1\geq3.$$
The proof of the Claim is completed.

\medskip

After a further iterate, we can make further reductions. Set $W_1(z)=V(z)$. For $n\geq1$, define $W_{n+1}(z)=W_n(z)^l\cdot V(z^{l^n} W_n(z)^k)$ inductively. Then $W_n(z)\in\C[z]\setminus\C$ and $W_n(0)\neq0$ for all $n\geq1$. By induction, it is clear that for every $n\geq1$,
$$\tilde{h}_1^{\circ n}(z)=z^{l^n}\cdot W_n(z^{k_1})^{k_1^\prime}\quad\text{and}\quad\tilde{h}_2^{\circ n}(z)=\zeta^{1+\cdots+l^{n-1}}\cdot z^{l^n}\cdot W_n(z^{k_2})^{k_2^\prime}$$
Since $\gcd(d,k)=1$ and $l=d^N-k\deg(V)$, we have $\gcd(l,k)=1$. By Euler's theorem, for the integer $k_2\varphi(k)\in\Z_{>0}$, we have
$$l^{k_2\varphi(k)}\equiv1\Mod{k}\quad\text{and}\quad1+\cdots+ l^{k_2\varphi(k)-1}\equiv0\Mod{k_2}.$$
Let $n_0\geq1$ be the minimal integer such that
$$l^{n_0}\equiv1\Mod{k}\quad\text{and}\quad1+\cdots+ l^{n_0-1}\equiv0\Mod{k_2}.$$
Replacing $(h_i,\tilde{h}_i,N,l,V,\zeta)$ by $(h_i^{\circ n_0},\tilde{h}_i^{\circ n_0},n_0 N,l^{n_0},W_{n_0},1)$ ($i=1,2$), we may assume $l\equiv1\Mod{k}$ and $\zeta=1$. We still assume $h_1\not\sim h_2$ and hence $\tilde{h}_1\not\sim\tilde{h}_2$. In particular, $\tilde{h}_1\neq\tilde{h}_2$, so $k_1\neq k_2$. This concludes the proof.
\end{proof}

\begin{Rem}\label{Remminequiv}
	Pakovich \cite[Theorem~1.6]{Paksemidecomp} (see also \cite[Theorem~3.44]{FG22}) showed that for every polynomial $P\in\C[z]$ of degree $d\geq2$, there exist $\pi_\min\in\C[z]\setminus\C$ and a polynomial $P_\min\in\C[z]$ of degree
	$d$ with $P\geq_{\pi_\min}P_\min$ satisfying the following universal property: for any polynomial $Q\in\C[z]$ with $P\geq Q$, there exist $\pi,\omega\in\C[z]\setminus\C$ such that $P\geq_\omega Q\geq_\pi P_\min$ and $\pi_\min=\pi\circ\omega$. Furthermore, $\deg(\pi_\min)$ is bounded by a constant depending only on $d$.
	
	From the proof above, it follows immediately that in Theorem~\ref{polynoninj} \eqref{polynoninj2}, the polynomials $(\phi_1(t)^{\circ N})_\min$ and $(\phi_2(t)^{\circ N})_\min$ are mutually semi-conjugate, hence
	$$(\phi_1(t)^{\circ N})_\min\sim (\phi_2(t)^{\circ N})_\min$$
	by \cite[Theorem~1.5]{Paksemidecomp}.
\end{Rem}

\begin{Rem}\label{Remmainpf}
Let $K$ be a number field such that $\phi_1$, $\phi_2$, and $C$ are defined over $K$. By \cite[Remark~5.21]{FG22}, the positive integer $N$ in \eqref{polynoninj1} or \eqref{polynoninj2} of Theorem~\ref{polynoninj} is bounded by a constant depending only on $[K:\Q]$ and $d$, as $n_0\leq k_2\varphi(k)$.
\end{Rem}

\begin{proof}[Proof of Theorem~\ref{intertwined}]
By restricting to a non-empty affine subset $C_0\subseteq C$, we can define a section
$$\psi:C_0\to\Poly^d\times\Poly^d$$
of
$$\Psi\times\Psi:\Poly^d\times\Poly^d\to\MPoly^d\times\MPoly^d$$
over $C_0$. By Theorem~\ref{polynoninj}, there exists a finite subset $S\subset C_0(\C)$ such that for every $t\in C_0(\C)\setminus S$, $\phi_1(t):=\pi_1(\psi(t))$ and $\phi_2(t):=\pi_2(\psi(t))$ satisfy \eqref{polynoninj1} or \eqref{polynoninj2} of Theorem~\ref{polynoninj}, where $\pi_j:\Poly^d\times\Poly^d\to\Poly^d$ is the $j$-th projection for $1\leq j\leq2$.

By \cite[Theorem~3.39(4)(5) and Proposition~3.41(3)]{FG22}, $f_t$ and $g_t$ are intertwined if they satisfy Theorem~\ref{polynoninj} \eqref{polynoninj1}. Note that for all $Q\in\C[z]\setminus\{0\}$ and integers $m,n\geq0$, the polynomials $z^m Q(z^n)$ and $z^m Q(z)^n$ are intertwined \cite[Theorem~3.39]{FG22}. Again by \cite[Theorem~3.39]{FG22}, we conclude that $f_t$ and $g_t$ are intertwined if they satisfy Theorem~\ref{polynoninj} \eqref{polynoninj2}. As $S\cup(C\setminus C_0)$ is finite, for all but finitely many $t=([f_t],[g_t])\in C(\C)$, $f_t$ and $g_t$ are intertwined. Ji and Xie \cite[Remark~1.6]{JX23} showed that the intertwined locus is Zariski-closed in $\sM_d\times\sM_d$, hence also Zariski-closed in $\MPoly^d\times\MPoly^d$. Therefore, for every $t=([f_t],[g_t])\in C(\C)$, $f_t$ and $g_t$ are intertwined.
\end{proof}

\subsection{Ritt moves}\label{S3.3}
Case \eqref{polynoninj2} of Theorem~\ref{polynoninj} involves polynomials arising from the ``Ritt moves'' in the Ritt theory of polynomial decompositions \cite{Ritt}; see \cite{RittZM} and \cite[\S~3.5.1]{FG22}. Up to composition with degree one polynomials, the only solutions of the equation $P\circ Q=\hat{P}\circ\hat{Q}$ in indecomposable $P,Q,\hat{P},\hat{Q}\in\C[z]$ (of degree $\geq2$) are the trivial solution $(P,Q)=(\hat{P},\hat{Q})$ and the nontrivial solutions of the forms:
\begin{align}
	\label{RM} z^n\circ z^s R(z^n)&=z^s R(z)^n\circ z^n,\\
	\label{Tcomm} T_p\circ T_q&=T_q\circ T_p,
\end{align}
where $R\in\C[z]\setminus\{0\}$, integers $s\geq0$ and $n\geq2$ are coprime, and $p,q$ are distinct primes. Such a replacement of $(P,Q)$ by $(\hat{P},\hat{Q})$ is referred to as a \emph{Ritt move} (up to composing with degree one polynomials). In this article, working modulo composition with maps in $\Aff(\C)$, we consider only the case \eqref{RM} and call the pair $(z^s R(z^n),z^s R(z)^n)$ a \emph{Ritt move}, as the trivial case is uninteresting, and \eqref{Tcomm} involves exceptional Chebyshev polynomials. We may also consider \eqref{RM} when $P,Q,\hat{P},\hat{Q}$ are not necessarily indecomposable.

With additional hypotheses, we can still deduce that the multiplier spectra coincide over an (infinite) arithmetic progression in case \eqref{polynoninj2} of Theorem~\ref{polynoninj}, as follows:
\begin{Thm}\label{Rittmove}
	Let $r,k\in\Z_{>0}$ with $r\geq2$ and $R(z)\in\C[z]\setminus\{0\}$. Set
	$$P(z)=z^r R(z^k)\quad\text{and}\quad Q(z)=z^r R(z)^k.$$
	Assume that $\gcd(r(r^M-1),k)=1$ for some $M\in\Z_{>0}$. Then there exist integers $d>c_1>0$ such that for every $N\in\Z_{\geq0}$,
	$$S_{c_1+Nd}(P)=S_{c_1+Nd}(Q).$$ 
\end{Thm}
\begin{Rem}\label{c1d}
	The conclusion trivially holds if $k=1$. When $k\geq2$, according to the proof of Theorem~\ref{Rittmove}, we can take $c_1$ to be the minimal positive integer such that $\gcd(r(r^{c_1}-1),k)=1$ and $d$ to be the minimal positive integer such that $r^d\equiv1\Mod{k}$.
\end{Rem}
\begin{proof}
If $k=1$, then $P=Q$ and the conclusion is trivial. Assume $k\geq2$. Write $R(z)=z^l R_0(z)$, where $l\in\Z_{\geq0}$ and $R_0(z)\in\C[z]$ with $R_0(0)\neq0$. Set $r_0=r+lk$. For $M\in\Z_{>0}$ such that $\gcd(r(r^M-1),k)=1$, we also have $\gcd(r_0(r_0^M-1),k)=1$. Replacing $(r,R)$ by $(r_0,R_0)$, we may assume $R(0)\neq0$. If $\deg(R)=0$, then both $P$ and $Q$ are conjugate to $z^r$, and the conclusion is immediate. Henceforth assume $s:=\deg(R)\geq1$.

Recall that for every rational map $f\in\C(z)$ of degree at least $2$ and a fixed point $z_0\in\Fix(f)$, we have $\rho_f(z_0)\neq1$ if and only if the multiplicity of $z_0$ in $\Fix(f)$ is $1$. For every polynomial $f(z)\in\C[z]$ of degree $m\geq2$ and every $n\in\Z_{>0}$, the point $\infty$ always has multiplier $\rho_{f^{\circ n}}(\infty)=0$, hence $\infty$ has multiplicity $1$ in $\Fix(f^{\circ n})$. Thus for polynomials, we only need to consider periodic points in $\C$. Let $\Fix(\cdot)$ denote $\Fix(\cdot,\C)$ from now on.

Fix an arbitrary integer $n\geq1$. From the commutative relation
\begin{equation}\label{commn}
	Q^{\circ n}\circ z^k=z^k\circ P^{\circ n},
\end{equation}
it is easy to see that the map
$$h:\Fix(P^{\circ n})\to\Fix(Q^{\circ n}),\quad z_0\mapsto z_0^k$$
is well-defined (as a map between sets).

Set $F_1=\Fix(P^{\circ n})$ and $F_2=\Fix(Q^{\circ n})$, which are multisets of cardinality $m^n$, where $4\leq m:=r+sk=\deg(P)=\deg(Q)$. For a multiset $X$ and an element $x\in X$, let $\multi_X(x)$ denote the multiplicity of $x$ in $X$. We claim that for every $z_0\in F_1$, we have
\begin{equation}\label{multisame}
	\multi_{F_1}(z_0)=\multi_{F_2}(z_0^k)\quad\text{and}\quad\rho_{P^{\circ n}}(z_0)=\rho_{Q^{\circ n}}(z_0^k).
\end{equation}

Differentiating \eqref{commn} and evaluating at $z_0\in\Fix(P^{\circ n})$, we obtain
\begin{equation}\label{diffeva}
	\left((Q^{\circ n})^\prime(z_0^k)-(P^{\circ n})^\prime(z_0)\right)\cdot z_0^{k-1}=0.
\end{equation}

Let $z_0\in\Fix(P^{\circ n})$ such that $z_0\neq0$. By \eqref{diffeva},
$$\rho_{P^{\circ n}}(z_0)=(P^{\circ n})^\prime(z_0)=(Q^{\circ n})^\prime(z_0^k)=\rho_{Q^{\circ n}}(z_0^k).$$
Thus \eqref{multisame} holds when $\rho_{P^{\circ n}}(z_0)=\rho_{Q^{\circ n}}(z_0^k)\neq1$. Assume $\rho_{P^{\circ n}}(z_0)=1$. Writing $\zeta_k=\exp(2\pi i/k)$, we have
\begin{equation}\label{eqex}
	Q^{\circ n}(z^k)-z^k=P^{\circ n}(z)^k-z^k=(P^{\circ n}(z)-z)\prod_{j=1}^{k-1}(P^{\circ n}(z)-\zeta_k^j z).
\end{equation}
The left-hand side of \eqref{eqex} can be regarded as a polynomial in the variable $z^k-z_0^k$, and the right-hand side as a polynomial in $z-z_0$. Since $z_0\neq0$, we have $\ord_{z-z_0}(z^k-z_0^k)=1$ and $\ord_{z-z_0}(P^{\circ n}(z)-\zeta_k^j z)=0$ for $1\leq j\leq k-1$. Applying $\ord$ to both sides of \eqref{eqex}, we obtain
\begin{align*}
	\ord_{z-z_0^k}(Q^{\circ n}(z)-z)&=\ord_{z^k-z_0^k}(Q^{\circ n}(z^k)-z^k)\\
	&=\ord_{z^k-z_0^k}(Q^{\circ n}(z^k)-z^k)\cdot\ord_{z-z_0}(z^k-z_0^k)\\
	&=\ord_{z-z_0}(Q^{\circ n}(z^k)-z^k)\\
	&=\ord_{z-z_0}(P^{\circ n}(z)^k-z^k)\\
	&=\ord_{z-z_0}(P^{\circ n}(z)-z).
\end{align*}
Thus $\multi_{F_1}(z_0)=\multi_{F_2}(z_0^k)\geq2$ and \eqref{multisame} holds.

For the point $0=h(0)$, since $r\geq2$, $0$ is a superattracting fixed point of both $P^{\circ n}$ and $Q^{\circ n}$, i.e., $\rho_{P^{\circ n}}(0)=\rho_{Q^{\circ n}}(0)=0$. Hence $\multi_{F_1}(0)=\multi_{F_2}(0)=1$, concluding the proof of \eqref{multisame}.

By \eqref{multisame} and $\# F_1=\# F_2$, if $h(z_0)=h(z_1)$ (in $\C$) implies $z_0=z_1$ (in $\C$) for $z_0,z_1\in F_1$, then $S_n(P)=S_n(Q)$. In this case, we say that the integer $n$ is a \emph{good} integer.

\textbf{Claim:} for every integer $n\geq1$, either $n$ is good, or $\gcd(r^n-1,k)>1$.

\emph{Proof of the Claim:} Assume $n\in\Z_{>0}$ is not good. Then there exist $z_0\neq z_1$ in $F_1$ with $z_1^k=z_0^k$. Hence $z_1\neq0$ and $z_1=\zeta z_0$ for some $\zeta\in\C$ with $\zeta^k=1$ and $\zeta\neq1$. It is easy to see that there exists $R_n(z)\in\C[z]$ with $R_n(0)\neq0$ such that $P^{\circ n}(z)=z^{r^n} R_n(z^k)$ and $Q^{\circ n}(z)=z^{r^n} R_n(z)^k$. From $P^{\circ n}(z_0)=z_0\neq0$ and $P^{\circ n}(\zeta z_0)=\zeta z_0\neq0$, we deduce $\zeta^{r^n}=\zeta$, so
$$1<\ord_{\C^\times}(\zeta)\mid\gcd(r^n-1,k).$$

Let $M$ be the minimal positive integer such that $\gcd(r(r^M-1),k)=1$. In particular, $r$ and $k$ are coprime. By Euler’s theorem, $r^{\varphi(k)}\equiv1\Mod{k}$. Let $d$ be the minimal positive integer with $r^d\equiv 1\Mod{k}$. Then $d\mid\varphi(k)$. For every integer $N\geq0$, we have
$$r^{M+Nd}-1\equiv r^M-1\Mod{k},$$
so $\gcd(r^{M+Nd}-1,k)=1$. By the above Claim, every integer in the arithmetic progression $(M+Nd)_{N=0}^\infty$ is good. Since the sequence $(r^n\bmod k)_n$ has exact period $d$, $M\leq d$ by minimality. If $M=d$, then
$$k\mid\gcd(r^d-1,k)=\gcd(r^M-1,k)=1,$$
contradicting the assumption that $k\geq2$. Thus $M<d$
\end{proof}

\begin{Eg}
Let $k\geq2$ be an integer, and $V(z)\in\C[z]$ a polynomial of degree $n\geq1$ with $V(0)\neq0$. Consider the Ritt move
$$(P(z):=z\cdot V(z^k),Q(z):=z\cdot V(z)^k).$$
We show that $S_1(P)\neq S_1(Q)$ for general such $V$. Write
$$V(z)=a_n z^n+\cdots+a_1 z+a_0\text{, where }(a_n,\dots,a_0)\in\C^\times\times\C^{n-1}\times\C^\times.$$
It is easy to compute that $S_1(Q)$ is the multiset
$$\left\lbrace0,\rho_Q(0)=a_0^k,\rho_Q(w_1)=1+kw_1\frac{V^{\prime}(w_1)}{V(w_1)},\dots,\rho_Q(w_{nk})=1+kw_{nk}\frac{V^{\prime}(w_{nk})}{V(w_{nk})}\right\rbrace$$
of cardinality $nk+2$, where we write $V(z)^k=1+a_n^k(z-w_1)\cdots(z-w_{nk})$. One can check that
$$\prod_{j=1}^{nk}\left(a_0-1-k\,w_j\frac{V^{\prime}(w_j)}{V(w_j)}\right)=\prod_{j=1}^{nk}\left(a_0-1-k\,w_j V(w_j)^{k-1}V^{\prime}(w_j)\right)$$
is a nonzero polynomial in $a_0,\dots,a_n$; denote it by $f(a_0,\dots,a_n)$.

When $(a_n,\dots,a_0)$ does not lie on the union of the hypersurfaces $a_0=a_0^k$ and $f(a_0,\dots,a_n)=0$ in $\left(\A^1_\C\setminus\{0\}\right)\times\A_\C^{n-1}\times\left(\A^1_\C\setminus\{0\}\right)$, we have $\rho_P(0)=a_0\notin S_1(Q)$, hence $S_1(P)\neq S_1(Q)$. Therefore, for $(a_n,\dots,a_0)$ in a Zariski-dense open subset of $\left(\A^1_\C\setminus\{0\}\right)\times\A_\C^{n-1}\times\left(\A^1_\C\setminus\{0\}\right)$, we have $S_1(P)\neq S_1(Q)$.
\end{Eg}
\begin{Eg}\label{23move}
	Let $P(z)=z^2(z^3+1)$ and $Q(z)=z^2(z+1)^3$, which satisfy the hypothesis of Theorem~\ref{Rittmove}. Indeed, by Remark~\ref{c1d}, for every $n\in\Z_{>0}$, we have
	$$S_{2n-1}(P)=S_{2n-1}(Q).$$
	Using computer computations, one finds that $S_2(P)\neq S_2(Q)$ and $S_4(P)\neq S_4(Q)$, so $P$ and $Q$ are not isospectral.
\end{Eg}
\begin{Eg}
	Consider the Ritt move $(P(z)=z(z^2-3),Q(z)=z(z-3)^2)$. Then $P(z-2)+2=Q(z)$, hence $P$ and $Q$ are conjugate and isospectral.
\end{Eg}
Beyond the coincidence of the multiplier spectra over an infinite arithmetic progression under additional hypotheses, we can consider the relation of the multiplier spectra of Ritt moves up to multiplicative dependence:
\begin{Prop}\label{multidep}
	Let $r,k\in\Z_{>0}$ with $r\geq2$ and $R(z)\in\C[z]\setminus\{0\}$. Set
	$$P(z)=z^r R(z^k)\quad\text{and}\quad Q(z)=z^r R(z)^k.$$
	\begin{enumerate}[(1)]
		\item \label{multidep1} For every $z_0\in\Per(P)$, there exist $w_0\in\Per(Q)$ and $m\in\Z_{>0}$ such that $\rho_P(z_0)=\rho_Q(w_0)^m$ and $n_Q(w_0)\mid n_P(z_0)$.
		\item \label{multidep2} For every $w_0\in\Per(Q)$, there exist $z_0\in\Per(P)$ and $m\in\Z_{>0}$ such that $\rho_Q(w_0)=\rho_P(z_0)^m$ and $n_Q(w_0)\mid n_P(z_0)$.
	\end{enumerate}
	In fact, for both \eqref{multidep1} and \eqref{multidep2} we can choose $m=n_P(z_0)/n_Q(w_0)$.
\end{Prop}
\begin{proof}
\eqref{multidep1} Let $z_0\in\Per(P)$ with exact period $n$. If $\rho_P(z_0)=0$, then take $w_0=\infty$ and any $m\in\Z_{>0}$. Assume $\rho_P(z_0)\neq0$. In particular, $z_0\notin\{0,\infty\}$ since $r\geq2$. From \eqref{diffeva} (which holds without assuming $\gcd(r(r^M-1),k)=1$), we have
$$n_Q(z_0^k)\mid n\quad\text{and}\quad\rho_{Q^{\circ n}}(z_0^k)=\rho_{P^{\circ n}}(z_0)=\rho_P(z_0).$$
Set $w_0=z_0^k$ and $m=n/n_Q(w_0)$.

\eqref{multidep2} Let $w_0\in\Per(Q)$ with exact period $n$. If $\rho_Q(w_0)=0$, take $z_0=\infty$ and any $m\in\Z_{>0}$. Assume $\rho_Q(w_0)\neq0$. In particular, $w_0\notin\{0,\infty\}$. By \cite[Lemma~3.42]{FG22}, there exists $z_0\in\PrePer(P)$ such that $z_0^k=w_0$. Replacing $z_0$ by $P^{\circ l}(z_0)$ for some suitable $l\in\Z_{>0}$ if necessary, we may assume that $z_0\in\Per(P)$. It is clear that $n=n_Q(w_0)\mid n_P(z_0)$. By \eqref{diffeva}, the choices $z_0$ and $m=n_P(z_0)/n$ satisfy the requirement.
\end{proof}
\begin{Rem}
We can give more explicit descriptions of $m$ and relations between the multiplier spectra $(S_n(P))$ and $(S_n(Q))$. For example, in the proof of the $\rho_Q(w_0)\neq0$ case of \eqref{multidep2}, we must have
$$P^{\circ jn}(z_0)=\zeta_j z_0\quad (0\leq j\leq m-1),$$
where $1=\zeta_0,\dots,\zeta_{m-1}$ are distinct $k$-th roots of unity. In particular,
$$m=\frac{n_P(z_0)}{n_Q(w_0)}\leq k.$$
We could also deduce some other restrictions involving $\zeta_j$. Similarly, we can require that $m=\frac{n_P(z_0)}{n_Q(w_0)}\leq k$ in \eqref{multidep1}.
\end{Rem}
Proposition~\ref{multidep} shows that for the Ritt move $(z^rR(z^k),z^r R(z)^k)$ with $r\geq2$, the multipliers of $P$ and $Q$ are mutually multiplicatively dependent. The author has studied multipliers modulo multiplicative dependence in the joint work \cite{JXZ} with Ji and Xie. For every rational map $f\in\C(z)$ of degree at least $2$, let $V(f)\subseteq\R$ be the $\Q$-vector space generated by $\{\chi_f(z)\mid z\in\Per^*(f)\}$.
\begin{Cor}\label{RittCor}
	Let $r,k\in\Z_{>0}$ with $r\geq2$ and $R(z)\in\C[z]\setminus\{0\}$. Set
	$$P(z)=z^r R(z^k)\quad\text{and}\quad Q(z)=z^r R(z)^k.$$
	Then the following hold:
	\begin{enumerate}[(1)]
		\item \label{RittCor1} $V(P)=V(Q)$;
		\item \label{RittCor2} $P$ is exceptional if and only if $Q$ is exceptional;
		\item \label{RittCor3} $P$ is PCF if and only if $Q$ is PCF.
	\end{enumerate}
\end{Cor}
\begin{proof}
\eqref{RittCor1} This follows from Proposition~\ref{multidep} and the definition of $V(\cdot)$.

\eqref{RittCor2} This follows from \eqref{RittCor1} and \cite[Theorem~1.5]{JXZ}. (Note that the converse of \cite[Theorem~1.5]{JXZ} also holds; see \cite{milnor2006lattes}.)

\eqref{RittCor3} This follows from Proposition~\ref{multidep} and \cite[Theorem~1.14]{JXZ}.
\end{proof}
\begin{Rem}
Part \eqref{RittCor2} of Corollary~\ref{RittCor} is a direct corollary of \cite[Theorem~4.4]{Paksemidecomp} (see also \cite[Theorem~3.39]{FG22}), which asserts that for two polynomials $f,g\in\C[z]$ of degree $d\geq2$ with $f\geq g$, the following statements hold:
\begin{itemize}
	\item $f$ is conjugate to $z^d$ if and only if $g$ is conjugate to $z^d$;
	\item $f$ is conjugate to $\pm T_d$ if and only if $g$ is conjugate to $\pm T_d$.
\end{itemize}
Recall that up to conjugacy, exceptional polynomials of degree $d$ are exactly $z^d$ and $\pm T_d(z)$, and that $T_d$ is conjugate to $-T_d$ if and only if $d$ is even.
\end{Rem}
\begin{Rem}
	Let $f,g\in\C(z)$ be two rational maps of the same degree $d\geq2$. Assume first that $f\geq g$. Arguing similarly as in the proof of Proposition~\ref{multidep}, it is easy to see that
	\begin{equation}\label{findimdiff}
		V(f)\cap V(g)+W=V(f)+V(g)
	\end{equation}
	for some finite-dimensional subspace $W$ of $V(f)+V(g)$. Consequently, for intertwined $f$ and $g$, this conclusion \eqref{findimdiff} still holds. Combined with \cite[Theorem~1.4]{JXZ} (whose converse also holds \cite{milnor2006lattes}), we obtain a proof of the fact that for intertwined $f$ and $g$, $f$ is exceptional if and only if $g$ is exceptional.
\end{Rem}

\section{Further directions and problems}\label{S4}
\subsection{Multiplier spectrum over an arithmetic progression}\label{S4.1}
We could ask to what extent the converse of Theorem~\ref{Rittmove} holds.
\begin{Que}\label{Q1}
	Let $P(z),Q(z)\in\C[z]$ be (non-exceptional) polynomials of the same degree $m\geq2$. Under what conditions do there exist integers $d>c_1>0$ such that
	$$S_{c_1+Nd}(P)=S_{c_1+Nd}(Q)$$
	for all $N\in\Z_{\geq0}$? Are there examples that do not arise from Theorem~\ref{Rittmove}, up to equivalence?
\end{Que}
Note that in Theorem~\ref{Rittmove} and Question~\ref{Q1}, the first term $c_1$ of the arithmetic progression is strictly less than the common difference $d$. If $c_1\geq d$, then we can take $k\in\Z_{>0}$ with $c_1<dk$ and consider the arithmetic progression $(c_1+Ndk)_{N=0}^\infty$ instead.

It is interesting to consider the case $c_1=d$, which relates to the notion of a stable multiplier spectrum:
\begin{Def}[Stable multiplier spectrum]
	Let $f,g\in\C(z)$ be two rational maps of the same degree $d\geq2$. We say that $f$ and $g$ have the same \emph{stable multiplier spectrum} and they are \emph{stably isospectral} if $\tau_{d^k}(f^{\circ k})=\tau_{d^k}(g^{\circ k})$ for some $k\in\Z_{>0}$.
\end{Def}
\begin{Rem}
	Clearly, $f$ and $g$ are stably isospectral if and only if there exists an arithmetic progression $A=(Nk)_{N=1}^\infty$ with the first term and common difference equal to an integer $k>0$ such that $S_n(f)=S_n(g)$ for all $n\in A$. Such an arithmetic progression $A=(Nk)_{N=1}^\infty\subseteq\Z_{>0}$ is called \emph{equidistant}.
\end{Rem}
\begin{Eg}
	If there exists $k\in\Z_{>0}$ such that $f^{\circ k}=g^{\circ k}$, then $f$ and $g$ are stably isospectral. For example, for $V(z)\in\C[z]\setminus\C$ and $n\in\Z_{>0}$, the polynomials $P(z)=zV(z^{2n})$ and $Q(z)=-zV(z^{2n})$ are stably isospectral, since $P^{\circ2}=Q^{\circ2}$.
\end{Eg}
Fix an integer $d\geq2$. Let
$$E_d:=\{([f],[g])\in\sM_d(\C)\times\sM_d(\C)\colon\exists k\in\Z_{>0},\tau_{d^k}(f^{\circ k})=\tau_{d^k}(g^{\circ k})\}.$$
Write $E_d=\cup_{k\geq1}E_{d,k}$ as a countable union, where
$$E_{d,k}:=\{([f],[g])\in\sM_d(\C)\times\sM_d(\C)\colon\tau_{d^k}(f^{\circ k})=\tau_{d^k}(g^{\circ k})\}.$$ Note that $E_{d,k}$ is the inverse image of $R_{d^k,m_{d^k}}$ (see \S\ref{S1.1}) under the morphism
$$\sM_d(\C)\times\sM_d(\C)\to\sM_{d^k}(\C)\times\sM_{d^k}(\C),([f],[g])\mapsto([f^{\circ k}],[g^{\circ k}]);$$
so $E_{d,k}$ is Zariski closed. For positive integers $k_1$ and $k_2$ with $k_1\mid k_2$, we have $E_{d,k_1}\subseteq E_{d,k_2}$. Hence $E_d=\cup_{n\geq1}E_{d,n!}$ is a countable union of increasing Zariski closed subsets of $\sM_d(\C)\times\sM_d(\C)$. Currently, it is unclear whether $E_d$ is Zariski closed.
\begin{Que}\label{Q2}
	For $d\geq2$, is $E_d$ Zariski closed in $\sM_d(\C)\times\sM_d(\C)$?
\end{Que}
\begin{Rem}
	The notion of stable multiplier spectrum and Question~\ref{Q2} were proposed by Prof. Xie. We learned about these ideas from Prof. Xie through private communications.
\end{Rem}
Replacing ``multiplier'' with ``length'' in the above discussion yields the analogous notion of a \emph{stable length spectrum}.

We show that the multiplier spectrum morphism over an equidistant arithmetic progression remains generically injective on the moduli space:
\begin{Prop}
	Let $d\geq2$ and $n\geq1$ be integers. The morphism
	$$\tau_d[n]:\sM_d(\C)\to\prod_{j=1}^{m_{d^n}}\A^{N_{d^n,j}}(\C),[f]\mapsto\tau_{d^n}([f^{\circ n}])$$
	is generically injective. Similarly, the restriction $\tilde{\tau}_d[n]$ of $\tau_d[n]$ to $\MPoly^d(\C)$ is also generically injective.
\end{Prop}
\begin{proof}
	Let $\kappa_{d,n}:\sM_d\to\sM_{d^n}$ be the composite morphism $[f]\mapsto[f^{\circ n}]$. Then $\tau_d[n]=\tau_{d^n}\circ\kappa_{d,n}$. By Theorem~\ref{ratgeninj}, $\tau_{d^n}$ is generically injective. Thus, it suffices to show that $\kappa_{d,n}$ is also generically injective, which follows directly from \cite[Theorem~5.4]{Pakdeg23}.
	
	To show the generic injectivity of $\tilde{\tau}_d[n]$, replacing Theorem~\ref{ratgeninj} with Theorem~\ref{polygeninj} in the above argument, it suffices to show that the composite morphism
	$$\tilde{\kappa}_{d,n}:\MPoly^d\to\MPoly^{d^n},[f]\mapsto[f^{\circ n}]$$
	is generically injective. If $d=2$, then Theorem~\ref{gendeg2}~\ref{gendeg2.2} implies the generic injectivity of $\tilde{\kappa}_{2,n}$. Assume now $d\geq3$. By Lemmas~\ref{genpresim}~and~\ref{genG1}, there exists a non-empty Zariski open subset $U_d\subseteq\MPoly^d(\C)$ such that for every $[f]\in U_d$, $f$ is pre-simple and $G(f)=1$. We claim that $\tilde{\kappa}_{d,n}\mid_{U_d}$ is injective. Let $[f],[g]\in U_d$ such that $\tilde{\kappa}_{d,n}([f])=\tilde{\kappa}_{d,n}([g])$, i.e., $[f^{\circ n}]=[g^{\circ n}]$ in $\MPoly^d(\C)$. After conjugating $g$ by a suitable degree one polynomial if necessary, we may assume $f^{\circ n}=g^{\circ n}$. Since $g$ is pre-simple, Corollary~\ref{Corpresimdecom} implies that $\tau\circ f=g=f\circ\sigma$ for some $\sigma,\tau\in\Aff(\C)$. Then $\sigma=1$ because $G(f)=1$, hence $f=g$. Thus, the morphism $\tilde{\kappa}_{d,n}$ is injective on $U_d$. Therefore, $\tilde{\tau}_d[n]$ is generically injective.
\end{proof}
The above theorem immediately implies the following corollary:
\begin{Cor}
	Let $d\geq2$ be an integer. Then the stable multiplier map on $\sM_d(\C)$ is very generically injective in the sense that there exists a subset $V\subseteq\sM_d(\C)$ of the form
	$$V=\sM_d(\C)\setminus\left(\cup_{m=1}^\infty Z_m\right),$$
	where each $Z_m$ is a proper Zariski closed subset of $\sM_d(\C)$, such that for all $[f],[g]\in V$, if $f$ and $g$ are stably isospectral, then $[f]=[g]$. An analogous result holds for polynomials as well.
\end{Cor}
Note that generalized Latt\`es maps are sparse in the moduli space, and hence so are maps arising from Ritt moves. For further study, we may ask whether the multiplier spectrum over any infinite arithmetic progression is generically injective on the moduli space.

\subsection{Non-injective locus and length spectrum}\label{S4.2}
We aim to obtain a more precise description of the non-injective locus of $\tilde{\tau}_d$ than that given in Theorem~\ref{polynoninj}, in light of Conjecture~\ref{ratnoninj}.

There are many directions for possible generalizations of Theorem~\ref{polynoninj}. We list some of them as follows:
\begin{itemize}
\item Can we bound $N$ and $\#S$ in Theorem~\ref{polynoninj} (more) effectively and explicitly?
\item Can we exclude case \eqref{polynoninj2} in light of Conjecture~\ref{ratgeninj}?
\item Can we describe $\PNI_d$ more precisely (possibly ignoring zero-dimensional components)?
\item What can be said for isospectral rational maps (not just polynomials)?
\end{itemize}

There are fewer known results for the length spectrum compared to the multiplier spectrum. The length spectrum is more difficult to study because it contains less information and fails to be an (algebraic) morphism between schemes. A significant result \cite[Theorem~1.5]{Ji2023} proved by Ji and Xie in this direction is that apart from the flexible Latt\`es family, the length spectrum determines the conjugacy class of rational maps up to only finitely many choices. We propose the following questions regarding the length spectrum:
\begin{itemize}
	\item Let $d\geq2$ be an integer. Does there exist a non-empty Zariski open subset $U$ of $\sM_d(\C)$ such that every pair $(x,y)\in U\times U$ with $L(x)=L(y)$ must be of the forms $([f],[f])$ or $([f],[\overline{f}])$? Here $\overline{f}$ represents the rational map obtained by applying complex conjugation to all the coefficients of $f$. Note that a conjecture of Ji and Xie \cite[Conjecture~1.9]{JX23} implies that this question has a positive answer. We could ask the same question with the weaker requirement that $U$ has full (or just positive) Lebesgue measure (which may not be Zariski open). Similar questions can be asked for polynomials as well.
	\item Classify rational maps (or polynomials) with the same length spectrum.
\end{itemize}

\end{document}